\documentclass[11pt,reqno]{amsart}
\usepackage[utf8]{inputenc}
\usepackage[T1]{fontenc}
\usepackage{lmodern}
\usepackage[english]{babel}
\usepackage{amsmath,a4wide}
\usepackage{xfrac}
\usepackage{esint}
\usepackage{stmaryrd,mathrsfs,bm,amsthm,mathtools,yfonts,amssymb,color,braket,booktabs,graphicx,graphics,amsfonts}
\newcommand{\R}{\mathbb{R}}

\newcommand{\N}{\mathbb{N}}

\theoremstyle{plain}
\newtheorem{theorem}{Theorem}[section]
\newtheorem*{theorem*}{Theorem}
\newtheorem{lemma}[theorem]{Lemma}
\newtheorem*{lemma*}{Lemma}
\newtheorem{prop}[theorem]{Proposition}
\newtheorem*{prop*}{Proposition}

\newtheorem*{corollary*}{Corollary}
\theoremstyle{definition}

\newtheorem*{example*}{e.g.}
\newtheorem{remark}[theorem]{Remark}
\newtheorem*{remark*}{Remark}

\newtheorem*{assumption*}{Assumption}
\numberwithin{equation}{section}
\newtheorem{definition}[theorem]{Definition}
\newtheorem*{definition*}{Definition}
\title{Existence of BV flow via elliptic regularization}
\author{Kiichi Tashiro}
\subjclass{53E10 (primary), 28A75}
\keywords{Elliptic regularization, geometric measure theory, mean curvature flow}
\address{(K.Tashiro) Department of Mathematics, Tokyo Institute of Technology, 2-12-1 Ookayama, Meguroku, Tokyo 152-8551, Japan}
\email{tashiro.k.ai@m.titech.ac.jp}

\begin{document}
	
	\maketitle
	\markboth{Kiichi Tashiro}{Existence of BV flow via elliptic regularization}
	
	\begin{abstract}
		We investigate a mean curvature flow obtained via elliptic regularization, and prove that it is not only a Brakke flow, but additionally a generalized BV flow proposed by Stuvard and Tonegawa. In particular, we show that the change in volume of the evolving phase can be expressed in terms of the generalized mean curvature of the Brakke flow.
	\end{abstract}
	
	\vspace{15mm}
	
	\section{Introduction}
	Arising as the natural $L^2$-gradient flow of the area functional, the mean curvature flow (henceforth referred to as MCF) is arguably one of the most fundamental geometric flows. 
 The unknown of MCF is a one-parameter family $ \{ M_t \}_{ t \geq 0 } $ of surfaces in the Euclidean space (or more generally some Riemannian manifold) such that the normal velocity vector $ V $ of $ M_t $ equals its mean curvature vector $ h $ at each point for every time, i.e.,
	\begin{equation}
		V = h\ \text{ on }\ M_t. \label{MCF}
	\end{equation}
	When given a compact smooth surface $ M_0 $, a unique smooth solution exists for a finite time until singularities such as shrinkage and neck pinching occur. Numerous frameworks of generalized solutions of MCF that allow singularities have been proposed and studied: we mention, among others, the Brakke flow \cite{brakke1978motion}, level set flow \cite{chen1991uniqueness,evans1991motion}, BV flow \cite{luckhaus1995implicit,laux2016convergence,laux2018convergence}, $ L^2 $ flow \cite{roger2008allen,bertini2017stochastic}, generalized BV flow \cite{StuvardTonegawa+2022}. These weak solutions have been investigated by numerous researchers in the last 40 years or so from varying viewpoints.
	\vskip.5\baselineskip
	The aim of the present paper is to show that the flow arising from elliptic regularization \cite{ilmanen1994elliptic} is a generalized BV flow in addition to being a Brakke flow. A generalized BV flow consists of a pair of phase function and Brakke flow, in many ways reminiscent to Ilmanen's enhanced motion described in \cite{ilmanen1994elliptic}, but the one which 
	relates the volume change of phase and the Brakke flow in an explicit manner. When an ($ n-1 $)-dimension smooth MCF $ \{ M_t \}_{ t \geq 0 } $ is the boundary of open sets $ \{ E_t \}_{ t \geq 0 } $, the following equality holds naturally for all $\phi \in C^1_c ( \R^n \times \R^+ )$	:
	\begin{equation}
		\frac{ d }{dt} \int_{ E_{ t }} \phi ( x , t )\ d x = \int_{ E_t } \partial_t \phi ( x , t )\ d x  +\int_{ M_t } \phi ( x , t )\ h ( x , t ) \cdot \nu_{ M_t } ( x )\ d \mathcal{ H }^{ n - 1 } ( x ) . \label{MCFcoareaformula}
	\end{equation}
	Here, $ h ( \cdot , t ) $ is the mean curvature of $M_t$ and $ \nu_{ M_t } $ is the unit outer normal vector of $M_t$. The notion of BV flow utilizes this equality to characterize the motion law, roughly speaking: a family of sets of finite perimeter $\{E_t\}_{t\geq 0}$ is a BV flow if the 
	reduced boundary $M_t:=\partial^* E_t$ has generalized mean curvature vector
	$h(\cdot,t)$ satisfying \eqref{MCFcoareaformula}. The underlying assumption of the BV
	flow is that the generalized mean curvature vector is derived from $\partial^* E_t$.
	The generalized BV flow is
	proposed by Stuvard and Tonegawa \cite{StuvardTonegawa+2022} so that the flow can
	allow possible integer multiplicities ($\geq 2$) for the underlying Brakke flow while still
	keeping the equality \eqref{MCFcoareaformula} for a sets of finite perimeter. When the 
	higher multiplicity portion of the Brakke flow has null measure, the generalized BV flow
	corresponds to a BV flow. 
	Additionally having this equality (\ref{MCFcoareaformula}) has a certain 
	conceptual advantage for the Brakke flow in that some non-uniqueness and
	stability issues of Brakke flow
	can be resolved: Fischer et al. \cite{fischer2020local} showed that a BV flow with smooth initial datum necessarily coincides with the smooth MCF until the time when some singularity appears. 
 There are some conditional existence results for BV flows such as \cite{luckhaus1995implicit,laux2016convergence,laux2018convergence} under an
	assumption that the approximate solutions converge to the limit without loss of surface energy. In the present paper, we prove without 
 any condition that the solution arising from elliptic regularization is 
	a generalized BV solution with possible higher integer multiplicities on the side of Brakke flow. 
	\vskip.5\baselineskip
	We next briefly mention closely related works on the MCF using elliptic regularization.
	Ilmanen's elliptic regularization \cite{ilmanen1994elliptic} 
	gives a Brakke flow by first minimizing weighted area functional with a parameter.
	The solution of this minimization is a translative soliton and smooth (except for a small set of singularity of codimension $7$), and by letting the
 parameter converges to $0$, one obtains a
 Brakke flow as a limit of the smooth flows. 
 One advantage of this method is that one
 can apply White's local
	regularity theorem for this Brakke flow arising from the elliptic
	regularization \cite{brian2005regularity}. Elliptic regularization has also been studied for constructing MCFs with the Neumann boundary conditions by Edelen \cite{Edelen+2020+95+137} and Dirichlet boundary conditions by White \cite{WhiteMCF2021}. In addition, by Schulze and White \cite{SchulzeWhite+2020+281+305}, this method is utilized to construct a MCF with a triple junction by setting up an appropriate minimizing problem within the class of flat chains. 
	\vskip.5\baselineskip
	The key element of the present paper is an estimate of $ L^2 $ boundedness
	of approximate velocity for elliptic regularization. The idea is to find the convergence of the velocity vector representing the motion of phase boundaries using the concept of measure function pairs by Hutchinson \cite{hutchinson1986second}. The existence of velocity leads to absolute continuity of phase boundary measure in space-time with respect to the weight
	measure of the Brakke flow. More precisely, if $ \{ \mu_t \} $ is the Brakke flow, and $ \{ E_t \} $ is the family of sets of finite perimeter driven by $ \mu_t $, we may obtain $ d| \partial S | \ll d \mu_t dt $, where $ S = \{ ( x , t ) \mid x \in E_t \} $
 is a space-time track of $ \{ E_t \} $, and $ | \partial S | $ is the total variation measure of the characteristic function of $S$ in space-time. Once this is done, we may recover the formula 
	\eqref{MCFcoareaformula} using a suitable version of co-area formula from geometric measure theory.
	\vskip.5\baselineskip
	The paper is organized as follows. In Section \ref{Preliminaries}, we set our notation and explain the main result. In Section \ref{EllipticRegularization}, we review an outline of elliptic regularization. In Section \ref{Proofofmainresult}, we construct the approximate velocity and show that the existence of velocity leads to a generalized BV flow, and then we prove that the limit
 of the translative soliton is indeed a generalized BV flow.
	
	\subsection*{Acknowledgment}
	The author would like to thank his supervisor Yoshihiro Tonegawa for his insightful feedback, careful reading, and improving the quality of this paper. The author was supported by JST, the establishment of university fellowships towards the creation of science technology innovation, Grant Number JPMJFS2112.
	
	\section{Preliminaries and Main Results}
	\label{Preliminaries}
	\subsection{Basic Notation}
	We shall use the same notation for the most part adopted in \cite[Section 2]{StuvardTonegawa+2022}.
	In particular, the ambient space we will be working in is the Euclidean space $ \R^n $
	or its open subset $U$, and $ \R^+ $ will denote the interval $ [ 0 , \infty ) $. The coordinates $ ( x , t ) $ are set in the product space $ \R^n \times \R $, and $ t $ will be thought of and referred to as ``time''. We will denote $ \mathbf{ p } $ and $ \mathbf{ q } $ the projections of $ \R^n \times \R $ onto its factor, so that $ \mathbf{ p } ( x , t ) = x $ and $ \mathbf{ q } ( x , t ) = t $. If $ A \subset \R^n $ is (Borel) measurable, $ \mathcal{L}^n(A) $ will denote the Lebesgue measure of $ A $, whereas $ \mathcal{H}^k (A) $ denotes the $ k $-dimensional Hausdorff measure of $ A $. When $ x \in \R^n $ and $ r > 0 $, $ B_r ( x ) $ denotes the closed ball centered at $ x $ with radius $ r $. More generally, if $ k $ is an integer, then $ B^k_r ( x ) $ will denote closed balls in $ \R^k $. The symbols $ \nabla , \nabla^{ \prime } , \Delta , \nabla^2 $ denote the spatial gradient and the full gradient in $ \R^n \times \R $, Laplacian, and Hessian, respectively. The symbol $ \partial_t $ will denote the time derivative.
	\vskip.5\baselineskip
	A positive Radon measure $ \mu $ on $\R^n $ (or ``space-time'' $\R^{n+1} $) is always also regarded as a positive linear functional on the space $ C^0_c ( \R^n) $ of continuous and compactly supported functions, with pairing denoted $ \mu ( \phi ) $ for $ \phi \in C^0_c (\R^n) $. The restriction of $ \mu $ to a Borel set $ A $ is denoted $ \mu \llcorner_A $, so that $ ( \mu \llcorner_A ) ( E ) := \mu ( A \cap E ) $ for any $ E \subset \R^n$. The support of $ \mu $ is denoted $ \mathrm{ supp }\, \mu $, and it is the closed set defined by
	\[
	\mathrm{ supp }\, \mu := \Bigl\{ x \in \R^n \mid \mu ( B_r ( x ) ) > 0 \text{ for all } r > 0\ \Bigr\}.
	\]
	For $ 1 \leq p \leq \infty $, the space of $ p $-integrable 
 functions with respect to $ \mu $ is denoted $ L^p ( \mu ) $.
 If $\mu=\mathcal L^n$, $ L^p (\mathcal L^n ) $
 is simply written $ L^p ( \R^n ) $. For a signed or vector-valued measure $ \mu $, $ | \mu | $ denotes its total variation. For two Radon measures $ \mu $ and $ \overline{\mu} $, when the measure $ \overline{\mu} $ is absolutely continuous with respect to $ \mu $, we write $ \overline{\mu} \ll \mu $.
	We say that a function $ f \in L^1 ( \R^n ) $ has a bounded variation, written $ f \in BV ( \R^n ) $, if
	\[
	\sup \left\{ \int_{\R^n} f \,\mathrm{ div } X\ dx\ \middle|\ X \in C^1_c ( \R^n ; \R^n ) ,\ \| X \|_{ C^0 } \leq 1 \right\} < \infty.
	\]
	If $ f \in BV ( \R^n ) $, then there exists an $ \R^n $-valued Radon measure (which we will call the measure derivative of $ f $ denoted by $ \nabla f $) satisfying
	\[
	\int_{\R^n} f \,\mathrm{ div } X\ dx = - \int_{\R^n} X \cdot d \nabla f\ \text{for all}\ X \in C^1_c ( \R^n ; \R^n).
	\]
	For a set $ E \subset \R^n $, $ \chi_E $ is the characteristic function of $ E $, defined by $ \chi_E =1$ if $ x \in E $ and $ \chi_E = 0 $ otherwise. We say that $ E $ has a finite perimeter if $ \chi_E \in BV(\R^n) $. When $ E $ is a set of finite perimeter, then the measure derivative $ \nabla \chi_E $ is the associated Gauss-Green measure, and its total variation $ | \nabla \chi_E | $ is the perimeter measure; by De Giorgi's structure theorem, $ | \nabla \chi_E | = \mathcal{ H }^{ n - 1 } \llcorner_{ \partial^* E } $, where $ \partial^* E $ is the reduced boundary of $ E $, and $ \nabla \chi_E = - \nu_E | \nabla \chi_E | = - \nu_E \mathcal{ H }^{ n - 1 } \llcorner_{ \partial^* E } $, where $ \nu_E $ is the outer pointing unit normal vector field to $ \partial^* E $.
	\vskip.5\baselineskip
	A subset $ M \subset \R^n $ is countably $ k $-rectifiable if it is $ \mathcal{ H }^k $-measurable and admits a covering 
	\[
	M \subset Z \cup \bigcup_{ i \in \N } f_i ( \R^k )
	\]
	where $ \mathcal{ H }^k ( Z ) = 0 $ and $ f_k : \R^k \to \R^n $ is Lipschitz. 
 If $M$ is countably $ k $-rectifiable and $\mathcal H^k(M)<\infty$, $M$ has a measure-theoretic tangent plane called approximate tangent plane for $\mathcal H^k$-a.e.$\,x\in M$ (\cite[Theorem 11.6]{simon1983lectures}),
 denoted by $T_x M$. We may simply refer to it as the tangent plane at $x\in M$ without fear of confusion. 
	A Radon measure $ \mu $ is said to be $ k $-rectifiable if there are a countably $ k $-rectifiable set $ M $ and a positive function $ \theta \in L^1( \mathcal{ H }^k \llcorner_{ M } ) $ such that $ \mu = \theta \mathcal{ H }^k \llcorner_{ M } $. This function $ \theta $ is called multiplicity of $ \mu $. The approximate tangent
	plane of $M$ in this case (which exists $\mu$-a.e.) is denoted by $T_x\,\mu$. 
	When $ \theta $ is an integer for $ \mu$-a.e., $ \mu $ is said to be integral. The first variation $ \delta \mu : C^1_c ( \R^n ; \R^n ) \to \R $ of a rectifiable Radon measure $ \mu $ is defined by
	\[
	\delta \mu ( X ) = \int_{\R^n } \mathrm{ div }_{ T_x\, \mu } X\ d\mu,
	\]
	where $ P_{ T_x \,\mu } $ is the orthogonal projection from $ \R^n $ to $ T_x\, \mu $, and $ \mathrm{ div }_{ T_x\, \mu } X = \mathrm{ tr } ( P_{ T_x\, \mu } \nabla X ) $. For an open set $ U \subset \R^n $, the total variation $ | \delta \mu | ( U ) $ of $ \mu $ is defined by
	\[
	| \delta \mu | ( U ) = \sup \left\{ \delta \mu ( X )\ \middle|\ X \in C^1_c ( U ; \R^n ) ,\ \| X \|_{C^0} \leq 1 \right\}.
	\]
	If the total variation $ | \delta \mu |(\tilde U)$ is finite for any bounded subset $\tilde U$ of $ U $, then $ \delta \mu $ is called locally bounded, and we can regard $ | \delta \mu | $ as a measure.\ If $ | \delta \mu | \ll \mu $, then the Radon--Nikod\'ym derivative (times $-1$) is called
	the generalized mean curvature vector $ h $ of $ \mu $, and we have
	\[
	\delta \mu ( X ) = - \int_{ \R^n } X \cdot h\ d\mu \quad \text{for all } X \in C^1_c ( \R^n; \R^n ).
	\]
	If $ \mu $ is integral, then $ h $ and $ T_x\, \mu $ are orthogonal for $ \mu $-a.e.~by
	Brakke's perpendicularity theorem \cite[Chapter 5]{brakke1978motion}.
	
	\subsection{Weak Notions of Mean Curvature Flow and main result}
	In this subsection, we introduce some weak solutions to the MCF. We briefly define and comment upon the three of interest in the present paper: We begin with the notion of Brakke flow introduced by Brakke \cite{brakke1978motion}.
	\begin{definition}
		A family of Radon measures $ \{  \mu_t \}_{ t \in \R^+ } $ in
  $\mathbb R^n$ 
  is \textit{an} ($ n - 1 $)-\textit{dimensional Brakke flow}                        
		if the following four conditions are satisfied:
		\begin{description}
			\item[\quad\textup{(1)}] For a.e. $ t \in \R^+ $, $\mu_t$ is integral and 
			$ \delta \mu_t $ is locally bounded and absolutely continuous with respect to $  \mu_t $ (thus the generalized mean curvature exists for a.e.~$t$, denoted by $h$).
			\item[\quad\textup{(2)}] For all $ s>0 $ and all compact set $ K \subset \R^n$, $ \sup_{ t \in [ 0 , s ] } \mu_t ( K ) < \infty $.
			\item[\quad\textup{(3)}] The generalized mean curvature $ h $ satisfies $ h \in L^2 (d\mu_tdt) $.
			\item[\quad\textup{(4)}] For all $ 0 \leq t_1 < t_2 < \infty $ and all test functions $ \phi \in C^1_c (\R^n \times \R^+ ; \R^+ )$,
			\begin{equation}
				\begin{split}
					\mu_{ t_2 } ( &\phi ( \cdot , t_2 ) ) - \mu_{ t_1 } ( \phi ( \cdot , t_1 ) )\\
					&\leq \int_{ t_1 }^{ t_2 } \int_{ \R^n} \Big( \nabla \phi ( x , t ) - \phi ( x , t )\,h ( x , t ) \Big) \cdot h ( x , t ) + \partial_t \phi ( x , t )\ d \mu_t ( x ) dt.
				\end{split}
				\label{Brakkeineq}
			\end{equation}
		\end{description}
		\label{Brakke}
	\end{definition}
	The inequality (\ref{Brakkeineq}) is motivated by the following identity,
	\begin{equation}
		\left.\int_{ M_t } \phi ( x , t )\ d\mathcal{ H }^{ n - 1 }\right|_{ t = t_1 }^{ t_2 } = \int_{ t_1 }^{ t_2 } \int_{ M_t } ( \nabla \phi - \phi\,h ) \cdot V + \partial_t \phi\ d\mathcal{ H }^{ n - 1 } dt \label{motivate}
	\end{equation}
	where $ M_t $ is an ($ n - 1 $)-dimensional smooth surface, $ h $ is the mean curvature vector, and $ V $ is the normal velocity vector of $ M_t $.\ In particular, if $\{M_t\}_{t\in[0,T)}$ is a smooth MCF (hence $V=h$), setting $  \mu_t := \mathcal{ H }^{ n - 1 } \llcorner_{ M_t } $ defines a Brakke flow for which (\ref{Brakkeineq}) is satisfied with the equality. Conversely, if $  \mu_t = \mathcal{ H }^{ n - 1 } \llcorner_{ M_t } $ with smooth $ M_t $ satisfies (\ref{Brakkeineq}), then one 
	can prove that $ \{ M_t \}_{ t \in [ 0 , T ) } $ is a classical solution to the MCF. The 
 notion of Brakke flow is equivalently (and 
 originally in \cite{brakke1978motion}) formulated in the framework of varifolds, but we use the above formulation using Radon measures, mainly for convenience. 
	\vskip.5\baselineskip
	The following definition of $ L^2 $ flow (modified slightly for our purpose) was given by Mugnai and R\"{o}ger \cite{roger2008allen}. 
	\begin{definition}[$ L^2 $ flow]
 A family of Radon measures $ \{  \mu_t \}_{ t \in \R^+ } $ in $\R^n$ is \textit{an} ($ n - 1 $)-\textit{dimensional $ L^2 $ flow} if it is satisfies (1)-(2) in Definition \ref{Brakke} as well as the following:
		\begin{description}
			\item[\quad\textup{(a)}] The generalized mean curvature $ h ( \cdot , t ) $ (which exists for a.e. $ t \in \R^+ $ by (1)) satisfies $ h ( \cdot , t ) \in L^2(  \mu_t ; \R^n ) $, and $ d \mu := d  \mu_t d t $ is a Radon measure on $ \R^n \times \R^+ $.
			\item[\quad\textup{(b)}] There exists a vector field $ V \in L^2 ( \mu ; \R^n ) $ and a constant $ C=C(\mu_t) > 0 $ such that
			\begin{description}
				\item[\quad\textup{(b'1)}] $ V ( x , t ) \perp T_x \, \mu_t $ for $ \mu $-a.e. $ ( x , t ) \in \R^n \times\R^+$,
				\item[\quad\textup{(b'2)}] For every test functions $ \phi \in C^1_c ( \R^n \times ( 0 , \infty ) ) $, it holds
				\begin{equation}
					\left| \int_{ 0 }^{ \infty } \int_{\R^n} \partial_t \phi ( x , t ) + \nabla \phi ( x , t ) \cdot V ( x , t )\ d \mu_t ( x ) dt\ \right| \leq C\,\| \phi \|_{ C^0 }.
					\label{L2flow}
				\end{equation}
			\end{description}
		\end{description}
		\label{L2def}
	\end{definition}
	The vector field $ V $ satisfying (\ref{L2flow}) is called the generalized velocity vector in the sense of $ L^2 $ flow. This definition interprets equality (\ref{motivate}) as a functional expression of the area change.
	\vskip.5\baselineskip
	Finally, we introduce the concept of generalized BV flow suggested by Stuvard and Tonegawa \cite{StuvardTonegawa+2022}.
	\begin{definition}[Generalized BV flow]
 Let $ \{ \mu_t \}_{ t \in \R^+ } $ and $ \{ E_t \}_{ t \in \R^+ } $ be families of Radon measures and sets of finite perimeter, respectively. The pair $ ( \{ \mu_t \}_{ t \in \R^+ } , \{ E_t \}_{ t \in \R^+} ) $ is \textit{a generalized BV flow} if all of the following hold:
		\begin{description}
			\item[\quad\textrm{(i)}] $ \{ \mu_t\}_{t\in\R^+} $ is a Brakke flow.
			\item[\quad\textrm{(ii)}] For all $ t \in \R^+$, $ | \nabla \chi_{ E_t } | \leq  \mu_t $.
			\item[\quad\textrm{(iii)}] For all $ 0 \leq t_1 < t_2 < \infty $ and all test functions $ \phi \in C^1_c ( \R^n \times \R^+ ) $,
			\begin{equation}
				\left.\int_{ E_t } \phi ( x , t )\ dx\ \right|^{ t_2 }_{ t = t_1 } = \int_{ t_1 }^{ t_2 } \int_{ E_t } \partial_t \phi ( x , t )\ dx dt + \int_{ t_1 }^{ t_2 } \int_{ \partial^* E_t } \phi ( x , t )\,h ( x , t )\cdot\nu_{E_t}(x)\, d\mathcal{ H }^{ n - 1 } ( x ) dt.
				\label{theareachange}
			\end{equation}
		\end{description}
		\label{generalizedBVflow}
	\end{definition}
	If $ \mu_t $ and $ E_t $ satisfy the above definition, we say ``$V=h$'' in the sense of generalized BV flow. This definition expresses that the interface $ \partial^* E_t $ is driven by the mean curvature of $\mu_t $. If $ \mu_t = | \nabla \chi_{ E_t } | $ for a.e.~$t$, the characterization \eqref{theareachange} coincides with the notion of BV flow considered by Luckhaus--Struzenhecker in \cite{luckhaus1995implicit} since
	the mean curvature of $\partial^* E_t$ is naturally defined to be $h(\cdot,t)$ in this case. On 
	the other hand, while the original BV flow is characterized only by \eqref{theareachange}, here $\mu_t$ is additionally a Brakke flow to which one can apply White's local regularity theorem \cite{brian2005regularity} (see \cite{kasai2014general,tonegawa2014second,stuvard2022endtime} for more general regularity theorems for Brakke flow).
	\begin{remark}
		Note that one can prove that Brakke flow is an $L^2$ flow of $V=h$ in general,
		but the opposite implication may not hold in general. The following is a simple counterexample. Define
		\[
		E_t =
		\begin{cases}
			\ \emptyset &( 0 \leq t < 1 )\\
			\ \left\{ x \in \R^n\ \middle|\ | x |^2 \leq 1 - 2 ( n - 1 ) ( t - 1 )\ \right\} &\left( 1 \leq t \leq 1 + \frac{ 1 }{ 2 ( n - 1 ) } \right),
		\end{cases}
		\]
		and consider $\mu_t:=\mathcal H^{n-1}\llcorner_{\partial E_t}$. One can show that it defines
		an $L^2$ flow with $V=h$ but it is not a Brakke flow. 
	\end{remark}
	The following claim is the essence of the main result of the present paper.
	\begin{theorem}
		\label{mainresult}
        Suppose that $E_0\subset \R^n$ is a set of finite
        perimeter and consider the initial value problem of MCF starting from $\partial^* E_0$. 
        Then, 
		the elliptic regularization of Ilmanan \cite{ilmanen1994elliptic} produces a generalized BV flow of $ V = h $. Namely, in addition to a
  Brakke flow $\{\mu_t\}_{t\in\R^+}$ with $\mu_0=\mathcal H^{n-1}\llcorner_{\partial^* E_0}$ (whose existence was
  proved in \cite{ilmanen1994elliptic}), there exists 
  a family of sets of finite perimeter $\{E_t\}_{t\in\R^+}$
  such that $(\{\mu_t\}_{t\in\R^+},\{E_t\}_{t\in\R^+})$
  is a generalized BV flow. 
	\end{theorem}
 In the present paper, we discuss the case 
 where the flow is considered in $\R^n$. Since the analysis
 is local in nature, the same characterization holds equally (with appropriate modifications) for flows in \cite{Edelen+2020+95+137,SchulzeWhite+2020+281+305,WhiteMCF2021}. 
	
	\section{Review on the existence of Brakke flow via elliptic regularization}
	\label{EllipticRegularization}
	In this section, we briefly review Ilmanen's proof on 
 the existence of Brakke flow in \cite{ilmanen1994elliptic} for the convenience of the reader (see also a lecture note on MCF by White \cite{WhiteMCFlecture2015}).
	\subsection{Translative Functional and Euler-Lagrange equation}
	\label{translative}
	The method of construction in \cite{ilmanen1994elliptic} uses the framework of rectifiable current and for general codimensional case. Since we are concerned only with the
 hypersurface case, we work with surfaces realized as boundaries of set of finite perimeter. In the following, the time variable is temporarily denoted as $ z $. Let $ \varepsilon > 0 $ be fixed. The symbol $ \mathbf{ e }_{ n + 1 } $ will denote the standard basis pointing the time direction, i.e., $ \mathbf{ e }_{ n + 1 } = ( 0 , \ldots , 0 , 1 ) \in \R^{n+1}$. We define the following functional for a set of finite perimeter $ S \subset \R^{n+1} $:
	\begin{equation}
		I^{ \varepsilon } ( S ) := \frac{ 1 }{ \varepsilon } \int_{ \R^{n+1} } \exp\Big( - \frac{ z }{ \varepsilon } \,\Big)\ d | \nabla^{ \prime } \chi_S | ( x , z ).
	\end{equation}
	We will say that a stationary point $ S^{ \varepsilon } $ of $ I^{ \varepsilon } $ is a translative soliton of the MCF with velocity $ - ( 1 / \varepsilon ) \mathbf{ e }_{ n + 1 } $. We consider a set of finite perimeter $ E_0 \subset \R^n $ as an initial datum and find a stationary point $ S^{ \varepsilon } $ for $ I^{ \varepsilon } $ by minimization. The existence
 of a minimizer follows from the compactness theorem of set of finite perimeter (see \cite[Section 3.2]{ilmanen1994elliptic}).
	\begin{lemma}
		Let $ E_0 \subset \R^n $ be a set of finite perimeter. Then there exists a set of finite perimeter $ S^{ \varepsilon } \subset \R^n \times \R $ such that
		\begin{description}
			\item[\quad\textup{(1)}] $ S^{ \varepsilon } \subset \R^n \times \R^+ $ and $\{z=0\}\cap \partial^* S^\varepsilon=E_0$,
			\item[\quad\textup{(2)}] $ I^{ \varepsilon } ( S^{ \varepsilon } ) \leq \mathcal{H}^{ n - 1 } ( \partial^* E_0 ) $,
			\item[\quad\textup{(3)}] $ S^{ \varepsilon } $ is a minimizer for $ I^{ \varepsilon } $ among the set
   $\tilde S\subset\R^n\times \R^+$ with 
   $\{z=0\}\cap\partial^*\tilde S=E_0$.
		\end{description}
		\label{Ellipticstationary}
	\end{lemma}
 \begin{definition}
     We define $ S_z = \{ x \in \R^n \mid ( x , z ) \in S \} $, that is, $S_z$ is the horizontal slice of $S$ at height $z$.
 \end{definition}
 The well-known regularity theory of geometric
 measure theory shows that the support of $|\nabla'\chi_{S^\varepsilon}|$ in $\{z>0\}$ 
 is a smooth hypersurface except for a 
 closed set of codimension $7$ (\cite{simon1983lectures}). 
	For such $ S^{ \varepsilon } $, we calculate the first variation of $ I^{ \varepsilon } $ and obtain the following equations (see \cite[Section 2.6]{ilmanen1994elliptic}).
	\begin{lemma}
		Let $S^\varepsilon$ be as in Lemma \ref{Ellipticstationary} and let $ h^\varepsilon $ denote the mean curvature of $ \partial^* S^{ \varepsilon } $ (computed as a submanifold in $\mathbb R^n\times(0,\infty)$). Then we have all of the following for $ | \nabla^{ \prime } \chi_{ S^{ \varepsilon } } | $-a.e. in $\{z>0\}$:
		\begin{description}
			\item[\quad\textup{(1)}] $ \varepsilon h^\varepsilon + P_{ T_{ ( x , z ) } ( \partial^* S^{ \varepsilon } )^{ \perp } } \left( \mathbf{ e }_{ n + 1 } \right) = 0 $,
			\item[\quad\textup{(2)}] $ | h^\varepsilon | \leq 1 / \varepsilon $,
			\item[\quad\textup{(3)}] $ P_{ T_{ ( x , z ) } ( \partial^* S^{ \varepsilon } )^{ \perp } } ( h^\varepsilon ) = h^\varepsilon $,
			\item[\quad\textup{(4)}] $ \varepsilon^2 | h^\varepsilon |^2 + | P_{ T_{ ( x , z ) } ( \partial^* S ) } ( \mathbf{ e }_{ n + 1 } ) |^2 = 1 $.
		\end{description}
		\label{EulerLagrange}
	\end{lemma}
	By considering the integration of this equation in the time direction, we obtain the following (see \cite[Section 4.5]{ilmanen1994elliptic}).
	\begin{lemma}
		For every $ 0 \leq z_1 < z_2 $,
		\begin{equation}
  \begin{split}
			\int_{ \partial^* S^{ \varepsilon }_{ z_2 } } | P_{ T_{ ( x , z_2 ) } ( \partial^* S^{ \varepsilon } ) } ( \mathbf{ e }_{ n + 1 } ) |\ d \mathcal{ H }^{ n - 1 } ( x ) &+ \int_{ \partial^* S^{ \varepsilon } \cap ( \R^n \times ( z_1 , z_2 ) ) }  \varepsilon | h^\varepsilon |^2 d \mathcal{ H }^n ( x , z ) \\
   &= \int_{ \partial^* S^{ \varepsilon }_{ z_1 } } | P_{ T_{ ( x , z_1 ) } ( \partial^* S^{ \varepsilon } ) } ( \mathbf{ e }_{ n + 1 } ) |\ d \mathcal{ H }^{ n - 1 } ( x ). 
   \end{split}
   \label{timeapriori}
		\end{equation}
		In particular,
		\begin{equation}
        \begin{split}
			\max\Biggl\{\sup_{z>0}\int_{ \partial^* S^{ \varepsilon }_{ z } } | P_{ T_{ ( x , z ) } ( \partial^* S^{ \varepsilon } ) } &( \mathbf{e}_{ n + 1 } ) |\ d \mathcal{ H }^{ n - 1 } (x),\\
            &\int_{ \partial^* S^{ \varepsilon } \cap (\R^n \times ( 0 , \infty )) } \varepsilon | h^\varepsilon |^2\ d \mathcal{ H }^n(x,z)\Biggr\} \leq \mathcal{ H }^{ n - 1 } ( \partial^* E_0 ). \label{initialbdd}
        \end{split}
		\end{equation}
		\label{zinibdd}
	\end{lemma}
	Now consider $ E^{ \varepsilon } = \kappa_{ \varepsilon } ( S^{ \varepsilon } ) $ in which $ S^{ \varepsilon } $ is shrunk by the map
 $ \kappa_{ \varepsilon } ( x , z ) = ( x , \varepsilon z ) $ in the $ z $ direction. Since $ \kappa_{ \varepsilon } $ is the contraction map by $ \varepsilon $ to $ z $, the determinant of Jacobian matrix of $ \kappa_{\varepsilon} $ on $ \partial^* S^{\varepsilon} $ is $ ( |P_{T_{(x,z)}(\partial^* S^{\varepsilon})} (\nabla^{\prime}\mathbf{p}(x,z))|^2 + \varepsilon^2 |P_{T_{(x,z)}(\partial^* S^{\varepsilon})^{\perp}}(\mathbf{e}_{n+1})|^2)^{1/2} $. Therefore, by Lemma \ref{zinibdd}, we have the following for the mass of $ \partial^* S^{\varepsilon } $ and $ \partial^* E^{\varepsilon} $ (see \cite[Section 5.1, 5.3]{ilmanen1994elliptic}).
	\begin{lemma}
		For any open interval $ A = ( a , b ) \subset \R^+ $, we obtain
		\begin{align}
            &| \nabla^{ \prime } \chi_{ S^{ \varepsilon } } | ( \R^n \times A ) \leq ( \mathcal{L}^1 ( A ) + \varepsilon )\,\mathcal{H}^{n-1} (\partial^* E_0), \label{themassbddS}\\
			&| \nabla^{ \prime } \chi_{ E^{ \varepsilon } } | ( \R^n \times A ) \leq \left( \mathcal{L}^1 ( A ) + \varepsilon^2 + ( \mathcal{L}^1 ( A ) + \varepsilon^2 )^{\frac12} \right) \mathcal{H}^{ n - 1 } ( \partial^* E_0 ). \label{themassbddE}
		\end{align}
		In particular, the result holds for any $ \mathcal{ L }^1 $-measurable set $ A
  \subset\R^+$ by approximation.
		\label{Eapriori}
	\end{lemma}
	For this $ S^{ \varepsilon } $, we define the following notation:
	\begin{equation}
		\sigma_{ - t / \varepsilon } ( x , z ) := \left( x , z - \frac{ t }{ \varepsilon } \right) ,\ S^{ \varepsilon } ( t ) := \sigma_{ - t / \varepsilon } ( S^{ \varepsilon } ) ,\ \mu_t^{ \varepsilon } := | \nabla^{ \prime } \chi_{ S^{ \varepsilon } ( t ) } |.
	\end{equation}
	This $ \mu_t^{ \varepsilon } $ is a Brakke flow on the $ ( x , z ) $ components of $ W^{ \varepsilon } := \{ ( x , z , t ) \in ( \R^n \times \R ) \times [ 0 , \infty ) \mid z > - t / \varepsilon \} $. Since the Brakke flow $ \mu_t^{ \varepsilon } $ satisfies $ \mu_t^{ \varepsilon } ( \R^n \times ( z_1 , z_2 ) ) \leq ( ( z_2 - z_1 ) + \varepsilon )\,\mathcal{ H }^{ n - 1 } ( \partial^* E_0 ) $ by (\ref{themassbddS}), we can apply the compactness theorem for Brakke flow to $ \mu_t^{ \varepsilon } $ \cite[Section 7.1]{ilmanen1994elliptic}. Thus taking a further subsequence from $ \mu_t^{ \varepsilon } $, there exists a Brakke flow $ \{ \overline{ \mu }_t \}_{ t \geq 0 } $ on the $ ( x , z ) $ components of $ W := ( \R^n \times \R \times ( 0 , \infty ) ) \cup ( \R^n \times ( 0 , \infty ) \times \{ 0 \} ) $ such that $ \mu_t^{ \varepsilon } $ converges to $ \overline{ \mu }_t $ as Radon measure. Since $ \overline{ \mu }_t $ is invariant to translations in the $ z $ direction, we have the following from the product lemma \cite[Lemma 8.5]{ilmanen1994elliptic}. See \cite[Section 8]{ilmanen1994elliptic} for the details of the above discussion.
	\begin{lemma}
		Let $ \theta \in C^2_c ( \R ; \R^+ ) $ with $ \int_{\R} \theta ( z )\ d z = 1 $ and $ \mathrm{ supp }\, \theta \subset ( 0 , \infty ) $ be fixed. We define a Radon measure $ \mu_t $ on $ \phi \in C^0_c ( \R^n ; \R^+ ) $ by
		\[
		\mu_t ( \phi ) := \overline{ \mu }_t ( \theta\,\phi ),
		\]
		then $ \mu_t $ is independent of $ \theta $ and the following hold:
		\begin{description}
			\item[\quad\textup{(1)}] $ \overline{ \mu }_t = \mu_t \otimes \mathcal{ L }^1 $ except for countable $ t \geq 0 $,
			\item[\quad\textup{(2)}] $ \{ \mu_t \}_{ t \geq 0 } $ is a Brakke flow on $ \R^n $.
		\end{description}
	\end{lemma}
	When applying the compactness theorem of Brakke flow, we take a further subsequence using the compactness of set of finite perimeter: there exists a set of finite perimeter $ E \subset \R^n \times \R^+ $ such that
	\[
	\chi_{ E^{ \varepsilon } } \to \chi_E \text{ in } L^1_{ loc } (\R^n\times\R^+),\ | \nabla^{ \prime } \chi_E | ( \phi ) \leq \liminf_{ \varepsilon \to +0 } | \nabla^{ \prime } \chi_{ E^{ \varepsilon } } | ( \phi ) \text{ for all } \phi \in C^0_c ( \R^n \times \R^+ ; \R^+ ).
	\]
	Since $ S^{ \varepsilon } ( t ) $ is the translation of $ S^{ \varepsilon } $ by $ - t / \varepsilon $ in the $ z $ direction, and $ E^{ \varepsilon } $ is the contraction by $ \varepsilon $ in the $ z $ direction as defined above, one can check the following:
	\begin{equation}
            \label{rescalingS}
	       S^{ \varepsilon } ( t )_z = S^{ \varepsilon }_{ z + t / \varepsilon } = E^{ \varepsilon }_{ t + \varepsilon z }.
	\end{equation}
	
	\section{A generalized BV flow: proof of Theorem \ref{mainresult}}
        \label{Proofofmainresult}
	The key to the proof is to construct an approximate velocity and obtain a suitable $ L^2 $ estimate for a convergence of the velocity.
	
	\subsection{Existence of measure-theoretic velocities}
	In this subsection, even after taking a subsequence, we use the same notation $ \varepsilon $ for simplicity. In establishing \eqref{generalizedBVflow},
 the main tool is the co-area formula applied 
 to $S^\varepsilon$. It is a formula describing the rate of change of the volume of $E_t$, and the obstacle to obtain
 this is the
 possible presence of a portion where $\partial^* S^\varepsilon$ is ``close to being horizontal'', that is, 
 we want to make sure that the domain $ \{ | P_{ T_{ ( x , z ) }( \partial^* S^{ \varepsilon } ) } ( \mathbf{ e }_{ n + 1 } ) | \approx 0 \} $ vanishes when passing to the limit $ \varepsilon \to + 0 $. According to Lemma \ref{EulerLagrange}(4), if $ | P_{ T_{ ( x , z ) } ( \partial^* S^{ \varepsilon } ) } ( \mathbf{ e }_{ n + 1 } ) | \approx 0 $, we have $ | h^\varepsilon |^2 \approx 1 / \varepsilon^2 $. Using this fact and the $ L^2 $ boundedness of $ h^\varepsilon $ (Lemma \ref{zinibdd}), we can prove that such a domain is vanishing as $ \varepsilon \to +0 $. In particular, since $ S^{\varepsilon} $ approaches a cylindrical shape (extended in the time direction), we can even show that $ | P_{T_{(x,z)}(\partial^* S^{\varepsilon})} (\mathbf{e}_{n+1}) |^2 \approx 1 $.
	\begin{lemma}
		For $ S^{ \varepsilon } $ of Section \ref{translative}, we define
		\[
		\Sigma^{ \varepsilon , k } := \left\{ ( x , z ) \in \partial^* S^{ \varepsilon }\ \middle|\ | P_{T_{(x,z)}(\partial^* S^{\varepsilon})} (\mathbf{e}_{n+1}) |^2 \leq 1-\frac{1}{k} \right\}.
		\]
		Then if $ 1 < k \leq \varepsilon^{ - 1/2 } $, we have $ \lim_{ \varepsilon \to +0 } | \nabla^{ \prime } \chi_{ S^{ \varepsilon } } | ( \Sigma^{ \varepsilon , k } ) = 0 $.
		\label{discrilemma}
	\end{lemma}
	\begin{proof}
        Let $ 1 < k \leq \varepsilon^{-1/2} $ be fixed. From Lemma \ref{EulerLagrange}(1) and $ 1 = | P_{T_{(x,z)}(\partial^* S^{\varepsilon})}(\mathbf{e}_{n+1}) |^2 + | P_{T_{(x,z)}(\partial^* S^{\varepsilon})^{\perp}}(\mathbf{e}_{n+1}) |^2 $, we have
        \[
            | \nabla^{\prime} \chi_{S^{\varepsilon}} | (\Sigma^{\varepsilon,k}) = | \nabla^{\prime} \chi_{S^{\varepsilon}} | \left( \left\{ (x,z) \in \partial^* S^{\varepsilon}\ \middle|\ \frac{1}{k\varepsilon^2} \leq |h^\varepsilon|^2 \right\} \right).
        \]
		Thus, by using Markov's inequality and the $ L^2 $ boundedness of $ h^\varepsilon $ (\ref{zinibdd}), we compute
		\begin{equation}
			| \nabla^{ \prime } \chi_{ S^{ \varepsilon } } | ( \Sigma^{ \varepsilon , k } )
                = | \nabla^{ \prime } \chi_{ S^{ \varepsilon } } | \left( \left\{ 1 \leq k \varepsilon^2 |h^\varepsilon|^2 \right\} \right)
				\leq k \varepsilon \int_{ \R^n \times \R } \varepsilon | h^\varepsilon |^2\ d | \nabla^{ \prime } \chi_{ S^{ \varepsilon } } |
				\leq \varepsilon^{ 1/2 }\,\mathcal{ H }^{ n - 1 } ( \partial^* E_0 ).
		\end{equation}
        By taking $ \varepsilon \to + 0 $, we obtain $ \lim_{ \varepsilon \to + 0 } | \nabla^{ \prime } \chi_{ S^{ \varepsilon } } | ( \Sigma^{ \varepsilon , k } ) = 0 $.
	\end{proof}
	The following two lemmas relate $ \mu_t $ and $ E_t $ so that these fit in the framework of generalized BV flow. 
	\begin{lemma}
    \label{EconvergenstoMlemma}
		Taking a further subsequence if necessary, 
  we have $ d | \nabla \chi_{ E^{ \varepsilon }_t } | d t \rightharpoonup d \mu_t d t $ as Radon measure.
	\end{lemma}
	\begin{proof}
		By (\ref{themassbddE}) and the co-area formula, we obtain
		\[
		\int_{ t_1 }^{ t_2 } \int_{ \R^n } d | \nabla \chi_{ E^{ \varepsilon }_t } | d t \leq | \nabla^{ \prime } \chi_{ E^{ \varepsilon } } | ( \R^n \times ( t_1 , t_2 ) ) \leq \left( ( t_2 - t_1 ) + \varepsilon^2 + ((t_2-t_1)+\varepsilon^2)^{\frac12}\right)\mathcal{ H }^{ n - 1 } ( \partial^* E_0 )
		\]
		for all $ 0 \leq t_1 < t_2 < \infty $. Thus, by the compactness theorem of Radon measure, $ d|\nabla \chi_{E^{\varepsilon}_t} | dt $ converges to some Radon measure.
		\par
		Next, we fix $ \phi \in C^0_c ( \R^n \times \R^+ ) $ and $ z \geq 0 $. Parallel translating with respect to time, for sufficiently small $ \varepsilon > 0 $, we have
		\begin{equation}
			\begin{split}
				\Biggl| \int_0^{ \infty } \int_{ \R^n } \phi\ d | \nabla \chi_{ E^{ \varepsilon }_t } | d t &- \int_0^{ \infty } \int_{ \R^n } \phi\ d | \nabla \chi_{ E^{ \varepsilon }_{ t + \varepsilon z } } | d t\ \Biggr|\\
				&\leq \int_0^{ \infty } \int_{ \R^n } | \phi ( x , t - \varepsilon z ) - \phi ( x , t ) |\ d | \nabla \chi_{ E^{ \varepsilon }_t } | d t\\
				&\leq \sup_{ ( x , t ) \in \R^n \times \R^+ } | \phi ( x , t - \varepsilon z ) - \phi ( x , t ) |\,| \nabla^{ \prime } \chi_{ E^{ \varepsilon } } | ( K ),
			\end{split}
		\end{equation}
		where $ K $ is a sufficiently large compact set for $ \phi $ and we used the co-area formula. Therefore, by letting $ \varepsilon \to +0 $ in the above, we can deduce
        \begin{equation}
            \label{translationmeasureconvergence}
            \lim_{ \varepsilon \to +0 } \int_0^{ \infty } \int_{ \R^n } \phi\ d | \nabla \chi_{ E^{ \varepsilon }_t } | d t = \lim_{ \varepsilon \to +0 } \int_0^{ \infty } \int_{ \R^n } \phi\ d | \nabla \chi_{ E^{ \varepsilon }_{ t + \varepsilon z } } | d t.
        \end{equation}
		\par
		Finally, we prove $ d | \nabla \chi_{ E^{ \varepsilon }_t } | d t \rightharpoonup d \mu_t d t $ as Radon measure. Let $ \phi \in C^0_c ( \R^n \times \R^+ ) $, and let $ \theta \in C^2_c ( \R ; \R^+ ) $ with $ \int_{ \R } \theta ( z )\,d z = 1 $ and $ \mathrm{ supp }\,\theta \subset ( 0 , \infty ) $ be arbitrary. To use the co-area formula for $ \partial^* S^{\varepsilon} (t) $, we translate $ \int_0^{ \infty } \int_{ \R^n } \phi\ d \mu_t d t $ as
		\begin{equation}
			\begin{split}
				&\int_0^{ \infty } \int_{ \R^n } \phi\ d  \mu_t d t
				= \lim_{ \varepsilon \to +0 } \int_0^{ \infty } \int_{ \R^{n+1} } \theta ( z )\,\phi ( x , t )\ d \mu_t^{ \varepsilon } ( x , z ) d t\\
				&= \lim_{ \varepsilon \to +0 } \Biggl\{\int_0^{ \infty }  \int_{ \R^{n+1} } \theta ( z )\,\phi ( x , t )\,( 1 - | P_{ T_{ ( x , z ) } ( \partial^* S^{ \varepsilon } ( t ) ) } ( \mathbf{ e }_{ n + 1 } ) | )\ d | \nabla^{ \prime } \chi_{ S^{ \varepsilon } ( t ) } | ( x , z ) d t\\
				&\quad \quad \quad \quad + \int_0^{ \infty }  \int_{ \R^{n+1} } \theta ( z )\,\phi ( x , t )\,| P_{ T_{ ( x , z ) } ( \partial^* S^{ \varepsilon } ( t ) ) } ( \mathbf{ e }_{ n + 1 } ) |\ d | \nabla^{ \prime } \chi_{ S^{ \varepsilon } ( t ) } | ( x , z ) d t\Biggr\}.
			\end{split}
			\label{A+B}
		\end{equation}
		Since $ | P_{T_{(x,z)}(\partial^* S^{\varepsilon})} (\mathbf{e}_{n+1}) |^2 \approx 1 $ from Lemma \ref{discrilemma}, setting $ A_{ \varepsilon } $ and $ B_{ \varepsilon } $ by
		\begin{align*}
			A_{ \varepsilon } &:= \int_0^{ \infty }  \int_{ \R^{n+1} } \theta ( z )\,\phi ( x , t )\,( 1 - | P_{ T_{ ( x , z ) } ( \partial^* S^{ \varepsilon } ( t ) ) } ( \mathbf{ e }_{ n + 1 } ) | )\ d | \nabla^{ \prime } \chi_{ S^{ \varepsilon } ( t ) } | ( x , z ) d t,\\
			B_{ \varepsilon } &:= \int_0^{ \infty } \int_{ \R^{n+1} } \theta ( z )\,\phi ( x , t )\,| P_{ T_{ ( x , z ) } ( \partial^* S^{ \varepsilon } ( t ) ) } ( \mathbf{ e }_{ n + 1 } ) |\ d | \nabla^{ \prime } \chi_{ S^{ \varepsilon } ( t ) } | ( x , z ) d t,
		\end{align*}
        we can predict $ A_{\varepsilon } \to 0 $ and $ B_{\varepsilon} \approx \int_0^{ \infty } \int_{\R^n} \phi\,d|\nabla \chi_{E^{\varepsilon}_t} | dt $.
        \par
		For $ B_{ \varepsilon } $, by the  co-area formula and (\ref{rescalingS}), we obtain
		\begin{equation}
			\begin{split}
				\lim_{ \varepsilon \to +0 } B_{ \varepsilon }
				&= \lim_{ \varepsilon \to +0 } \int_0^{ \infty } \int_{ \R^{n+1} } \theta\,\phi\ d | \nabla \chi_{ S^{ \varepsilon }_{ t/\varepsilon + z } } | d z d t\\
				&= \lim_{ \varepsilon \to +0 } \int_0^{ \infty } \int_{ \R^{n+1} } \theta\,\phi\ d | \nabla \chi_{ E^{ \varepsilon }_{ t + \varepsilon z } } | d z d t\\
				&= \lim_{ \varepsilon \to +0 } \int_0^{ \infty } \int_{ \R^n } \phi\ d | \nabla \chi_{ E^{ \varepsilon }_t } | d t.
			\end{split}
			\label{B}
		\end{equation}
        Here, we also used Fubini's theorem
        to change the order of integration with respect to $z$ and $t$, (\ref{translationmeasureconvergence}) and $ \int_{\R} \theta ( z )\ dz = 1 $ for the second line to the third line.
        \par
		Now we consider $ A_{ \varepsilon } $. For $ \Sigma^{ \varepsilon , k } $ of Lemma \ref{discrilemma}, $ 1 - | P_{ T_{ ( x , z ) } ( \partial^* S^{ \varepsilon } ) } ( \mathbf{ e }_{ n + 1 } ) |^2 < 1/k $ is satisfied for all $ ( x , z ) \in \partial^* S^{ \varepsilon } \backslash \Sigma^{ \varepsilon , k } $.
        Hence, by using the co-area formula, the mass boundedness of $ \partial^* S_z $ (\ref{initialbdd}) and 
        $ \int_{\R} \theta ( z )\,dz = 1 $, we calculate as
        \begin{equation*}
            \begin{split}
                &\int_0^{ \infty } \int_{ \R^{n+1} } | \theta\,\phi |\,\chi_{ \sigma_{ - t / \varepsilon }\,( \partial^* S^{ \varepsilon } \backslash \Sigma^{ \varepsilon , k } ) }\,( 1 - | P_{ T_{ ( x , z ) } ( \partial^* S^{ \varepsilon } ( t ) ) } ( \mathbf{ e }_{ n + 1 } ) |^2 )\ d | \nabla^{ \prime } \chi_{ S^{ \varepsilon } ( t ) } | d t\\
                &\leq \int_0^{ \infty } \int_{ \R^{n+1} } | \theta\,\phi |\,\chi_{ \sigma_{ - t / \varepsilon }\,( \partial^* S^{ \varepsilon } \backslash \Sigma^{ \varepsilon , k } ) }\,\frac{1}{k}\,\frac{| P_{ T_{ ( x , z ) } ( \partial^* S^{ \varepsilon } ( t ) ) } ( \mathbf{ e }_{ n + 1 } ) |^2}{1-\frac{1}{k}}\ d | \nabla^{ \prime } \chi_{ S^{ \varepsilon } ( t ) } | d t\\
                &\leq \frac{ \mathcal{L}^1(K)\,\sup | \phi | }{ k - 1 } \int_{\R} \int_{\R^n} \theta\,| P_{ T_{ ( x , z ) } ( \partial^* S^{ \varepsilon } ) } ( \mathbf{ e }_{ n + 1 } ) |\ d | \nabla \chi_{S^{\varepsilon}_z} | dz\\
                &\leq \frac{ \mathcal{L}^1(K)\,\sup | \phi | }{ k - 1 }\,\mathcal{H}^{n-1} (\partial^* E_0),
            \end{split}
        \end{equation*}
        where $ K $ is a sufficiently large bounded interval for $ \phi $. Therefore, by using Lemma \ref{discrilemma} for $ k = \varepsilon^{ - 1 / 2 } $, we obtain
		\begin{equation}
			\begin{split}
				| A_{ \varepsilon } |
				&\leq \int_0^{ \infty } \int_{ \R^{n+1} } |\theta\,\phi|\,(\chi_{ \sigma_{ - t / \varepsilon } ( \Sigma^{ \varepsilon , \varepsilon^{-1/2} } ) } + \chi_{ \sigma_{ - t / \varepsilon } ( \partial^* S^{ \varepsilon } \backslash \Sigma^{ \varepsilon , \varepsilon^{-1/2} } ) } )\\
                &\hspace{60pt} \times( 1 - | P_{ T_{ ( x , z ) } ( \partial^* S^{ \varepsilon } ( t ) ) } ( \mathbf{ e }_{ n + 1 } ) |^2 )\ d | \nabla^{ \prime } \chi_{ S^{ \varepsilon } ( t ) } | d t\\
				&\leq \int_0^{ \infty } \int_{ \R^{n+1} } | \theta\,\phi |\,\chi_{ \sigma_{ - t / \varepsilon } ( \Sigma^{ \varepsilon , \varepsilon^{-1/2} } ) }\ d | \nabla^{ \prime } \chi_{ S^{ \varepsilon } ( t ) } | d t\\
                &\quad + \int_0^{ \infty } \int_{ \R^{n+1} } | \theta\,\phi |\,\chi_{ \sigma_{ - t / \varepsilon }\,( \partial^* S^{ \varepsilon } \backslash \Sigma^{ \varepsilon , \varepsilon^{-1/2} } ) }\,( 1 - | P_{ T_{ ( x , z ) } ( \partial^* S^{ \varepsilon } ( t ) ) } ( \mathbf{ e }_{ n + 1 } ) |^2 )\ d | \nabla^{ \prime } \chi_{ S^{ \varepsilon } ( t ) } | d t\\
				&\leq C ( \theta , \phi , E_0 ) \left( \int_{ \R^{n+1} } \chi_{ \Sigma^{ \varepsilon , \varepsilon^{-1/2} } }\ d | \nabla^{ \prime } \chi_{ S^{ \varepsilon } } | + \frac{1}{\varepsilon^{-1/2}-1} \right)\\
				&\leq C ( \theta , \phi , E_0 )\,( | \nabla^{ \prime } \chi_{ S^{ \varepsilon } } | ( \Sigma^{ \varepsilon , \varepsilon^{ - 1/2 } } ) + \varepsilon^{ 1/2 } ),
			\end{split}
			\label{A}
		\end{equation}
		where $ C ( \theta , \phi , E_0 ) $ is a constant that depends only on $ \theta $, $ \phi $ and the initial value $ E_0 $. By Lemma \ref{discrilemma}, we obtain $ \lim_{ \varepsilon \to +0 } A_{ \varepsilon } = 0 $. Thanks to (\ref{A+B})-(\ref{A}), we have $ d | \nabla \chi_{ E^{ \varepsilon }_t } | d t \rightharpoonup d \mu_t d t $ as Radon measure.
	\end{proof}
	\begin{lemma}
		Taking a further subsequence if necessary, we have $ | \nabla \chi_{ E_t } | ( \phi ) \leq \mu_t ( \phi )$ for all $ \phi \in C_c^0 ( \R^n ; \R^+ ) $ and for a.e. $ t \geq 0 $.
		\label{E<<mu}
	\end{lemma}
	\begin{proof}
		From Section \ref{translative}, $ \chi_{ E^{ \varepsilon } } \to \chi_E $ in $ L^1_{ loc } ( \R^n \times \R^+ ) $. Let $ z > 0 $ be arbitrary. Since the parallel translation is continuous in $ L^1 $, we have $ \chi_{ E^{ \varepsilon } + \varepsilon z \mathbf{e}_{n+1} } \to \chi_E $ in $ L_{loc}^1 ( \R^n \times \R^+ ) $. Taking further subsequence from $ \{ E^{ \varepsilon } + \varepsilon z \mathbf{e}_{n+1} \} $ if necessary, $ \chi_{ E^{ \varepsilon } } ( x , t + \varepsilon z ) $ converges to $ \chi_{ E }( x , t ) $ as $ \varepsilon \to + 0 $ for $ \mathcal{ L }^{ n + 1 } $-a.e.\,$ ( x , t ) \in \R^n \times \R^+ $, and by Fubini, we have $ \chi_{ E^{ \varepsilon }_{ t + \varepsilon z } } \to \chi_{ E_t } $ in $ L_{ loc }^1 ( \R^n ) $ for a.e.\, $ t \geq 0 $. Thus we obtain $ | \nabla \chi_{ E_t } | ( \phi ) \leq \liminf_{ \varepsilon \to + 0 } | \nabla \chi_{ E^{ \varepsilon }_{ t + \varepsilon z } } | ( \phi ) $ for all $ \phi \in C_c^0 ( \R^n ; \R^+ ) $ and a.e.\,$ t \geq 0 $ by the lower semi-continuity of variation measure. Let $ \phi \in C^0_c ( \R^n ; \R^+ ) $ and $ \theta \in C^2_c ( \R ; \R^+ ) $ with $ \int_{ \R } \theta ( z )\,d z = 1 ,\ \mathrm{ supp }\, \theta \subset ( 0 , \infty ) $ be arbitrary. Then we obtain
		\begin{equation}
			\begin{split}
				\mu_t ( \phi )
                &= \overline{ \mu }_t ( \theta\,\phi )
				= \lim_{ \varepsilon \to +0 }\int_{ \R^{n+1}  } \theta\,\phi\ d | \nabla^{ \prime } \chi_{ S^{ \varepsilon } ( t ) } |\\
                &\geq \liminf_{ \varepsilon \to +0 } \int_{ \R^{n+1} } \theta\,\phi\,| P_{ T_{ ( x , z ) } ( \partial^* S ( t ) ) } ( \mathbf{e}_{ n + 1 } ) |\ d | \nabla^{ \prime } \chi_{ S^{ \varepsilon } ( t ) } |
				= \liminf_{ \varepsilon \to +0 } \int_{ \R } \theta \int_{ \R^n } \phi\ d | \nabla \chi_{ S^{ \varepsilon } ( t )_z } | d z\\
				&= \liminf_{ \varepsilon \to +0 } \int_{ \R } \theta \int_{ \R^n } \phi\ d | \nabla \chi_{ S^{ \varepsilon }_{ z + t/\varepsilon } } | d z
				= \liminf_{ \varepsilon \to +0 } \int_{ \R } \theta \int_{ \R^n } \phi\ d | \nabla \chi_{ E^{ \varepsilon }_{ t + \varepsilon z } } | d z\\
				&\geq \int_{\R} \theta\,\liminf_{ \varepsilon \to +0 } | \nabla \chi_{ E^{ \varepsilon }_{ t + \varepsilon z } } | ( \phi )\ d z
				\geq | \nabla \chi_{ E_t } | ( \phi )
			\end{split}
		\end{equation}
        where we used the co-area formula, Fatou's Lemma, the lower semi-continuity of $ | \nabla \chi_{ E_t } | $, $ \int_{ \R } \theta ( z )\,dz = 1 $ and (\ref{rescalingS}).
	\end{proof}
	Finally, we construct an approximate velocity to show that $  d| \nabla^{\prime} \chi_{ E } | \ll d \mu_t d t $. If the co-area formula is available, we obtain
	\begin{equation}
    \label{classicalcoarea}
	\int_E \partial_t \phi\ dx dt = \int_{ \partial^* E } \phi\,\mathbf{ q } ( \nu_E )\ d \mathcal{ H }^n = \int_{ 0 }^{ \infty } \int_{ \partial^* E_t } \phi\,\frac{ \mathbf{ q } ( \nu_E ) }{ | \mathbf{ p } ( \nu_E ) | }\,d \mathcal{ H }^{ n - 1 } d t
	\end{equation}
	for all $ \phi \in C^1_c ( \R^n \times ( 0 , \infty ) ) $. Since we have that the domain $ \{ | \mathbf{ p } ( \nu_{ S^{ \varepsilon } } ) | \approx 0 \} $ where the co-area formula is not applicable goes to measure 0 from Lemma \ref{discrilemma} and $ | \mathbf{p}(\nu_{S^{\varepsilon}}) | = | P_{T_{(x,z)}(\partial^* S)} (\mathbf{e}_{n+1}) | $, we may construct the approximate velocity based on
 (\ref{classicalcoarea}).
	\begin{prop}
		Taking a further subsequence if necessary, there exists $ V \in L^2 ( d \mu_t d t ) $ such that
		\begin{equation}\label{vel2}
			\int_E \partial_t \phi\ d x d t = - \int_0^{\infty} \int_{ \R^n } \phi\,V\ d  \mu_t d t,
		\end{equation}
		for all $ \phi \in C^1_c ( \R^n \times ( 0 , \infty ) ) $.
		\label{velocity}
	\end{prop}
	\begin{proof}
		To use Lemma \ref{discrilemma}, we assume that $ \varepsilon > 0 $ is sufficiently small. Let $ \phi \in C^1_c ( \R^n \times ( 0 , \infty ) ) $ be arbitrary. We define the approximate velocity of $ E^{ \varepsilon } $ by
		\[
		V_{ \varepsilon } ( x , t ):= - \chi_{ \kappa_{ \varepsilon } ( \partial^* S^{ \varepsilon } \backslash \Sigma^{ \varepsilon , 2 } ) }\,\frac{ \mathbf{ q } ( \nu_{ E^{ \varepsilon } } ) }{ | \mathbf{ p } ( \nu_{ E^{ \varepsilon } } ) | } ( x , t ) = \begin{cases}
			\ - \frac{ \mathbf{ q } ( \nu_{ E^{ \varepsilon } } ) }{ | \mathbf{ p } ( \nu_{ E^{ \varepsilon } } ) | } ( x , t ) &( ( x , t ) \in \kappa_{ \varepsilon } ( \partial^* S^{ \varepsilon } \backslash \Sigma^{ \varepsilon , 2 } ) )\\
			\ 0 &( ( x , t ) \in \kappa_{ \varepsilon } ( \Sigma^{ \varepsilon , 2 } ) ).
		\end{cases}
		\]
		Since the map $ \kappa_{ \varepsilon } $ 
  shrinks $z$-variable by $ \varepsilon $, the following holds as the relationship between the unit normal vectors of $ S^{ \varepsilon } $ and $ E^{ \varepsilon } $:
		\begin{equation}
        \label{resStoEnu}
		\frac{ \mathbf{ q } ( \nu_{ S^{ \varepsilon } } ) }{ | \mathbf{ p } ( \nu_{ S^{ \varepsilon } } ) | } ( x , z ) = \varepsilon\,\frac{ \mathbf{ q } ( \nu_{ E^{ \varepsilon } } ) }{ | \mathbf{ p } ( \nu_{ E^{ \varepsilon } } ) | } ( x , t ) ,\ t = \varepsilon z.
		\end{equation}
        Furthermore, since the area element of $ \kappa_{\varepsilon} $ is $ \varepsilon $ and $ t = \varepsilon z $, we have
        \begin{equation}
            \label{relationshipES}
            \int_{ E^{ \varepsilon } } \partial_t \phi\ d x d t
				= \int_{ S^{ \varepsilon } } \partial_z \phi\ d x d z.
        \end{equation}
		A simple geometric argument shows $ |\mathbf{p} ( \nu_{S^{\varepsilon}})| = |P_{T_{(x,z)}(\partial^* S^{\varepsilon})} (\mathbf{e}_{n+1})| $, and by the definition
  of $\Sigma^{\varepsilon,2}$, we may deduce $ 1/2 < | \mathbf{ p } ( \nu_{ S^{ \varepsilon } } ) |^2 $ on $ \partial^* S^{ \varepsilon } \backslash \Sigma^{ \varepsilon , 2 } $. Thus, by the co-area formula and (\ref{relationshipES}), we obtain
		\begin{equation}
			\begin{split}
				\int_{ E^{ \varepsilon } } \partial_t \phi\ d x d t
				&= \int_{ S^{ \varepsilon } } \partial_z \phi\ d x d z\\
				&= \int_{ \R^n \times (0,\infty) } ( \chi_{ \Sigma^{ \varepsilon , 2 } } + \chi_{ \partial^* S^{ \varepsilon } \backslash \Sigma^{ \varepsilon , 2 } } )\,\phi\,\mathbf{ q } ( \nu_{ S^{ \varepsilon } } )\ d | \nabla^{ \prime } \chi_{ S^{ \varepsilon } } |\\
				&= \int_{ \R^n \times (0,\infty) } \chi_{ \Sigma^{ \varepsilon , 2 } }\,\phi\,\mathbf{ q } ( \nu_{ S^{ \varepsilon } } )\ d | \nabla^{ \prime } \chi_{ S^{ \varepsilon } } | + \int_0^{\infty} \int_{ \R^n } \phi\,\chi_{ \partial^* S^{ \varepsilon } \backslash \Sigma^{ \varepsilon , 2 } }\,\frac{ \mathbf{ q } ( \nu_{ S^{ \varepsilon } } ) }{ | \mathbf{ p } ( \nu_{ S^{ \varepsilon } } ) | }\ d | \nabla \chi_{ S^{ \varepsilon }_z } | d z\\
				&= \int_{ \R^n \times (0,\infty) } \chi_{ \Sigma^{ \varepsilon , 2 } }\,\phi\,\mathbf{ q } ( \nu_{ S^{ \varepsilon } } )\ d | \nabla^{ \prime } \chi_{ S^{ \varepsilon } } | - \int_0^{\infty} \int_{ \R^n } \phi\,V_{ \varepsilon }\ d | \nabla \chi_{ E^{ \varepsilon }_t } | d t.
			\end{split}
		\end{equation}
		From Lemma \ref{discrilemma}, we see that
        \begin{equation}
            \label{firsttermconvergence}
            \lim_{\varepsilon\to+0} \int_{ \R^n \times (0,\infty) } \chi_{ \Sigma^{ \varepsilon , 2 } }\,\phi\,\mathbf{ q } ( \nu_{ S^{ \varepsilon } } )\ d | \nabla^{ \prime } \chi_{ S^{ \varepsilon } } | = 0.
        \end{equation}
        Next, we prove the following with respect to the second term:
		\begin{equation}
			\int_0^{\infty} \int_{ \R^n } | V_{ \varepsilon } |^2\ d | \nabla \chi_{ E^{ \varepsilon }_t } | d t \leq C < \infty, \label{apriorivelocity}
		\end{equation}
		where $ C $ is a constant that depends only on $\mathcal H^{n-1}(\partial^* E_0)$. By (\ref{resStoEnu}), we see that
        \begin{equation*}\begin{split}
        \int_0^{\infty} \int_{ \R^n } | V_{ \varepsilon } |^2\ d | \nabla \chi_{ E^{ \varepsilon }_t } | d t
        &= \int_0^{\infty} \int_{\R^n} \left( \frac{1}{\varepsilon}\,\chi_{\partial^* S^{\varepsilon} \backslash \Sigma^{\varepsilon,2}}\,\frac{ \mathbf{ q } ( \nu_{ S^{ \varepsilon } } ) }{ | \mathbf{ p } ( \nu_{ S^{ \varepsilon } } ) | } \right)^2 \varepsilon\ d | \nabla \chi_{ S^{\varepsilon}_z } | dz \\&
        = \frac{1}{\varepsilon} \int_0^{\infty} \int_{\R^n} \left( \chi_{\partial^* S^{\varepsilon} \backslash \Sigma^{\varepsilon,2}}\,\frac{ \mathbf{ q } ( \nu_{ S^{ \varepsilon } } ) }{ | \mathbf{ p } ( \nu_{ S^{ \varepsilon } } ) | }\right)^2 d | \nabla \chi_{ S^{\varepsilon}_z } | dz.
        \end{split}
        \end{equation*}
        From $ 1/2 < | \mathbf{ p } ( \nu_{ S^{ \varepsilon } } ) |^2 $ on $ \partial^* S^{ \varepsilon } \backslash \Sigma^{ \varepsilon , 2 } $ and Lemma \ref{EulerLagrange}(1), we obtain
		\[
		\left( \chi_{ \partial^* S^{ \varepsilon } \backslash \Sigma^{ \varepsilon , 2 } }\,\frac{ \mathbf{ q } ( \nu_{ S^{ \varepsilon } } ) }{ | \mathbf{ p } ( \nu_{ S^{ \varepsilon } } ) | } \right)^2 \leq 2\,\varepsilon^2\,| h^\varepsilon |^2.
		\]
        Thus, by (\ref{initialbdd}), we have
		\[
		\frac{ 1 }{ \varepsilon } \int_0^{\infty} \int_{ \R^n } \left( \chi_{ \partial^* S^{ \varepsilon } \backslash \Sigma^{ \varepsilon , 2 } }\,\frac{ \mathbf{ q } ( \nu_{ S^{ \varepsilon } } ) }{ | \mathbf{ p } ( \nu_{ S^{ \varepsilon } } ) | } \right)^2 d | \nabla \chi_{ S^{ \varepsilon }_z } | d z \leq 2 \int_0^{\infty} \int_{ \R^n } \varepsilon\,| h^\varepsilon |^2\ d | \nabla \chi_{ S^{ \varepsilon }_z } | d z \leq 2\,\mathcal{ H }^{ n - 1 } ( \partial^* E_0 ),
		\]
        and thus (\ref{apriorivelocity}) is proved. As Lemma \ref{EconvergenstoMlemma} and (\ref{apriorivelocity}) are valid, we can apply Theorem \ref{measurefunctiontheorem} to $ ( d | \nabla \chi_{ E^{ \varepsilon }_t } | dt, V_{\varepsilon} ) $. Therefore, taking further subsequence if necessary, we obtain a function $ V \in L^2 (d\mu_t dt) $ such that
        \begin{equation}
        \label{Vconvergence}
        \lim_{\varepsilon\to+0} \int_0^{\infty} \int_{\R^n} \phi\,V_{\varepsilon}\ d|\nabla\chi_{E^{\varepsilon}_t} | dt = \int_0^{\infty} \int_{\R^n} \phi\,V\ d\mu_t dt
        \end{equation}
        for all $ \phi \in C^1_c (\R^n \times (0,\infty) ) $. By (\ref{firsttermconvergence}), (\ref{Vconvergence}) and $ \chi_{ E^{ \varepsilon } } \to \chi_E $ in $ L^1_{loc} (\R^n\times\R) $, we obtain
		\begin{equation*}
        \begin{split}
		\int_E \partial_t \phi\ d x d t
        &= \lim_{ \varepsilon \to +0 } \int_{ E^{ \varepsilon } } \partial_t \phi\ d x d t\\
        &= \lim_{ \varepsilon \to +0 } \left( - \int_0^{\infty} \int_{ \R^n } \phi\,V_{ \varepsilon }\ d | \nabla \chi_{ E^{ \varepsilon }_t } | d t \right)
        = - \int_0^{\infty} \int_{ \R^n } \phi\,V\ d \mu_t d t,
        \end{split}
		\end{equation*}
		for all $ \phi \in C^1_c (\R^n\times(0,\infty)) $. This completes the proof.
	\end{proof}
     \begin{lemma}
        \label{areacontilemma}
        There exists $G\subset\R^+$ with $\mathcal L^1(
        \R^+\setminus G)=0$ such that $\chi_{E_t}$ is 
        $1/2$-H\"{o}lder continuous in $L^1(\R^n)$ norm
        with respect to $t$ on $G$.
    \end{lemma}
    \begin{proof}
        Let $G\subset\R^+$ be a set such that $t\in G$ is a Lebesgue point of function $f_\phi (s):=\int_{E_s}\phi\,dx$ for any $\phi\in C_c^1(\R^n)$. By choosing a countable dense set of functions in $C_c^1(\R^n)$ and using a standard
        result in measure theory, one can prove that such
        $G$ is a full-measure set in $\R^+$. Let $t_1,t_2$
        ($t_1<t_2$) be arbitrary points in $G$, and consider a smooth approximation $\eta$ of $\chi_{[t_1,t_2]}$. Use $\phi(x)\,\eta(t)$ in \eqref{vel2} and let $\eta\to \chi_{[t_1,t_2]}$ 
        to obtain 
        \[
        -\int_{E_{t_2}}\phi\ dx+\int_{E_{t_1}}\phi\ dx
        =-\int_{t_1}^{t_2}\int_{\R^n} \phi\,V\ d\mu_t dt.
        \]
        By approximation, we may replace $\phi$ by $\chi_{E_{t_1}}$ and obtain 
        \begin{equation}\label{ssa}
        \begin{split}
        |\mathcal L^n(E_{t_1})-\mathcal L^n(E_{t_2}\cap E_{t_1})|
        &\leq \left(\int_{t_1}^{t_2}\mu_t(E_{t_1})\ dt\right)^{\frac12}\|V\|_{L^2(d\mu_t dt)} \\
        &\leq (t_2-t_1)^{\frac12}\,\mathcal H^{n-1}(\partial^* E_0)^\frac12\,\|V\|_{L^2(d\mu_t dt)},
        \end{split}
        \end{equation}
        where we also used
        $\mu_t(\R^n)\leq \mu_0(\R^n)=\mathcal H^{n-1}(\partial^* E_0)$. This inequality 
        is due to the energy
        decreasing property of Brakke flow which follows
        from \eqref{Brakkeineq}. 
        The left-hand side of \eqref{ssa} is $\mathcal L^n(E_{t_1}\setminus E_{t_2})$. One can 
        obtain the similar estimate for $\mathcal L^n(E_{t_2}\setminus E_{t_1})$ by considering
        $\chi_{E_{t_2}}$. This proves the claim.
         \end{proof}  
    \begin{remark}
    \label{abstcontirmk}
    By Lemma \ref{areacontilemma}, if necessary, 
    we may re-define $E$
    so that $\chi_{E_t}$ is $1/2$-H\"{o}lder 
    continuous in $L^1(\R^n)$ on $\R^+$. We also point out that, 
    one can re-define the Brakke flow $\{\mu_t\}_{t\in\R^+}$ so that it is left-continuous
    at all $t\in\R^+$. This is because, for any 
    $\phi\in C_c^2(\R^n)$, $\mu_t(\phi)-C(\phi)\,t$ 
    is a monotone decreasing function of $t$ for a suitable
    $C>0$, and is discontinuous on a countable set at most
    (see for example \cite[Proposition 3.3]{tonegawa2019brakke}). At these 
    discontinuous points, one may re-define
    $\mu_t$ (by approaching from the left)
    so that it is left-continuous while 
    keeping \eqref{Brakkeineq}. Now 
        the claim of Lemma \ref{E<<mu} is for 
        a.e.$\,t\geq0$, while Definition \ref{generalizedBVflow}(ii) is for all $t\geq 0$. Let $\tilde G$ be the set of points where the conclusion of Lemma \ref{E<<mu} holds, 
        which is a full-measure set of $\R^+$. For any $t\notin \tilde G$, we may choose a sequence
        $\{t_i\}\subset \tilde G$ approaching from
        left to $t$. Since $\chi_{E_{t_i}}\to\chi_{E_t}$ in $L^1(\R^n)$, we have for any
        $\phi\in C^0_c(\R^n;\R^+)$
        \[|\nabla\chi_{E_t}|(\phi)\leq \liminf_{i\to\infty}|\nabla\chi_{E_{t_i}}|(\phi)\leq \liminf_{i\to\infty}
        \mu_{t_i}(\phi)=\mu_t(\phi).
        \]
        Here the first inequality is due to the lower semi-continuous property, the second is due to 
        $t_i\in \tilde G$, and the last is the left-continuity of $\mu_t$. 
        Thus we have the desired property Definition \ref{generalizedBVflow}(ii). 
    \end{remark} 
        
	\begin{prop}
        \label{S<<dmudtprop}
		For $ \{ \mu_t \}_{ t \geq 0 } $ and $ \{ E_t \}_{ t \geq 0 } $ of Section \ref{translative}, we have $ d| \nabla^{ \prime } \chi_{ E } | \ll d \mu_t d t $.
	\end{prop}
 
	\begin{proof}
		From Proposition \ref{velocity},
		\[
		\nabla^{ \prime } \chi_{ E } ( x , t ) = ( \nabla \chi_{ E_t } d t , V ( x , t )\,d \mu_t d t )
		\]
		is satisfied in the sense of vector-valued measure. Thus, from Lemma \ref{E<<mu}, we obtain $ d| \nabla^{ \prime } \chi_{ E } | \ll d \mu_t d t $.
	\end{proof}
	
	\subsection{Basic Properties of $ L^2 $ Flow and Set of Finite Perimeter}
	\label{Property}
	In this subsection, we state the properties of $ L^2 $ flow and set of (locally) finite perimeter. The proof of Theorem \ref{mainresult} will follow from those properties. The arguments in this subsection are mostly contained in \cite{roger2008allen,StuvardTonegawa+2022}
	and we include this for the convenience of the reader. 
	\begin{prop}
		Let $\{ \mu_t \}_{ t \in \R^+ } $ and $ E $ be as in Section \ref{translative} and let
  $d\mu=d\mu_t dt$.
		Then $ \mu \llcorner_{ \partial^* E } $ is a rectifiable Radon measure and we have the following for $ \mathcal{ H }^n $-a.e. $ ( x , t ) \in \partial^* E \cap \{t>0\}$:
		\begin{description}
			\item[\quad\textup{(1)}] the tangent space $ T_{ ( x , t ) }\, \mu $ exists, and $ T_{ ( x , t ) } \,\mu = T_{ ( x , t ) } ( \partial^* E ) $,
			\item[\quad\textup{(2)}] $ \begin{pmatrix}
				h ( x , t )\\
				1
			\end{pmatrix}
			\in T_{ ( x , t ) }\, \mu $,
			\item[\quad\textup{(3)}] $ x \in \partial^* E_t $, and $ T_x \, \mu_t = T_x ( \partial^* E_t ) $,
			\item[\quad\textup{(4)}] $ \mathbf{ p } ( \nu_E ( x , t ) ) \neq 0 $, and $ \nu_{ E_t } ( x ) = | \mathbf{ p } ( \nu_E ( x , t ) ) |^{ - 1 }\ \mathbf{ p } ( \nu_E ( x , t ) ) $,
			\item[\quad\textup{(5)}] $ T_x ( \partial^* E_t ) \times \{ 0 \} $ is linear subspace of $ T_{ ( x , t ) }\, \mu $.
		\end{description}
		\label{princi}	
	\end{prop}
	The crucial step of the proof of Theorem \ref{mainresult} is to prove the above proposition, for which the $ L^2 $ flow property of $ \mu_t $ plays a pivotal role, and this proposition is proved in detail by \cite[Lemma 4.7]{StuvardTonegawa+2022}. In this paper, we will give a brief outline of the proof of Proposition \ref{princi}. 
	\vskip.5\baselineskip
	First, the following are simple propositions of $ L^2 $ flow by \cite[Proposition 3.3]{roger2008allen} and \cite[Theorem 4.3]{StuvardTonegawa+2022}.
	\begin{prop}
		Let $\{ \mu_t \}_{ t \in \R^+ } $ and $ V $ be an $ L^2 $ flow in Definition \ref{L2def}, and let $ \mu $ be the space-time measure $ d \mu = d \mu_t d t $. Then,
		\begin{equation}
			\begin{pmatrix}
				V ( x , t )\\
				1
			\end{pmatrix} \in T_{ ( x , t ) }\, \mu
		\end{equation}
		at $ \mu $-a.e. $ ( x , t ) \in \R^n\times \R^+ $ wherever the tangent space $ T_{ ( x , t ) }\, \mu $ exists.
		\label{L2cor}
	   \end{prop}
        \begin{prop}
		\label{BtoL2}
		The Brakke flow $ \{ \mu_t \}_{ t\in\R^+ } $ in Definition \ref{Brakke} with $ \mu_0 (\R^n) < \infty $ is an $ L^2 $ flow with the velocity $ V = h $ in Definition \ref{L2def}. Namely, there exists $ C = C ( \mu_t ) > 0 $ such that
		\begin{equation}
			\left| \int_{ 0 }^{ \infty } \int_{ \R^n } \partial_t \phi ( x , t ) + \nabla \phi ( x , t ) \cdot h ( x , t )\ d \mu_t ( x ) dt\ \right| \leq C\,\| \phi \|_{ C^0 }, \label{L2ineq}
		\end{equation}
		for all $ \phi \in C^1_c ( \R^n \times ( 0 , \infty ) ) $.
	\end{prop}
        Next, before the proof of Proposition \ref{princi}, we will need some consequences of Huisken's monotonicity formula for MCF. Now we briefly state the consequences necessary to prove the main result. See \cite[Section 3.2]{tonegawa2019brakke} for discussion below in detail. First, we set some notation. For $ ( y , s ) \in \R^n \times \R^+ $, we define the backwards heat kernel $ \rho_{(y,s)} $ by
        \[
            \rho_{(y,s)}(x,t) := \frac{1}{(4\pi(s-t))^{(n-1)/2}} \exp\left( -\frac{|x-y|^2}{4(s-t)} \right) \quad \text{for all } 0 \leq t < s \text{ and } x \in \R^n,
        \]
        as well as the truncated kernel
        \[
            \hat{\rho}^r_{(y,s)}(x,t) := \eta\left(\frac{x-y}{r}\right)\,\rho_{(y,s)}(x,t),
        \]
        where $ r>0 $ and $ \eta \in C^{\infty}_c (B_2(0);\R^+) $ is a suitable cut-off function such that $ \eta \equiv 1 $ on $ B_1 (0) $, $ 0 \leq \eta \leq 1 $, $ |\nabla \eta| \leq 2 $ and $ \| \nabla^2 \eta \| \leq 4 $. The following is a variant of Huisken's monotonicity formula for MCF (for example, see \cite[Section 3.2]{tonegawa2019brakke} in detail).

        \begin{lemma}
            \label{Huiskenmonoto}
            Let $ \{ \mu_t \}_{ t \in \R^+ } $ is a Brakke flow in Definition \ref{Brakke}. Then there exists $ c ( n ) > 0 $ with the following property. For every $ 0 \leq t_1 < t_2 < s < \infty $, $ y \in \R^n $ and $ r > 0 $, it holds that
            \begin{equation}
                \left.\mu_t (\hat{\rho}^r_{(y,s)}(x,t))\right|^{t_2}_{t=t_1} \leq c ( n )\,\frac{ t_2 - t_1 }{ r^2 } \sup_{ t \in [ t_1 , t_2 ] } \frac{ \mu_t ( B_{2r} ) }{ r^{n-1} }.
            \end{equation}
        \end{lemma}
        
        As a consequence, Lemma \ref{Huiskenmonoto} and a local mass bounded (Definition \ref{Brakke}(2)) indicate the following, which provides the upper bound of mass density ratio (\cite[Proposition 3.5]{tonegawa2019brakke}).
        \begin{lemma}
            \label{upperdensityratio}
            Let $ \{ \mu_t \}_{ t \in \R^+ } $ is a Brakke flow in Definition \ref{Brakke}, and let $ d \mu = d\mu_t dt $. For any $ \delta > 0 $, $ x_0 \in \R^n $ and $ R > 0 $, there exists $ c (\delta,n,R) > 0 $ with the following property. For any $ t \in [ \delta , \infty ) $ and $ B_r(y) \subset B_R(x_0) $, we have
            \[
                \frac{\mu_t(B_r(y))}{r^{n-1}} \leq c(\delta,n,R) \sup_{s\in[0,t]}\mu_s(B_{3R}(x_0)).
            \]
            In particular, $ \Theta^{*n} (\mu,(x,t)) < \infty $ for all $ ( x , t ) \in \R^n \times (0,\infty) $ holds.
        \end{lemma}
	For the proof of Proposition \ref{princi}, we need the following general facts on sets of finite perimeter (\cite[Theorem 18.11]{maggi2012sets}).
	\begin{lemma}
		If $ E \subset \R^n \times \R $ is a set of locally finite perimeter, then the horizontal section
		\[
		E_t = \{ x \in \R^n \mid ( x , t ) \in E \}
		\]
		is a set of locally finite perimeter in $ \R^n $ for a.e.\,$t\in\mathbb R$, and the following properties hold:
		\begin{description}
			\item[\quad\textup{(1)}] $ \mathcal{ H }^{ n - 1 } ( \partial^* E_t \Delta ( \partial^* E )_t ) = 0 $,
			\item[\quad\textup{(2)}] $ \mathbf{ p } ( \nu_E ( x , t ) ) \neq 0 $ for $ \mathcal{ H }^{ n - 1 } $-a.e. $ x \in ( \partial^* E )_t $,
			\item[\quad\textup{(3)}] $ \nabla \chi_{ E_t } = | \mathbf{ p } ( \nu_E ( x , t ) ) |^{ - 1 } \mathbf{ p } ( \nu_E ( x , t ) )\,\mathcal{ H }^{ n - 1 } \llcorner_{ ( \partial^* E )_t } $.
		\end{description}
		\label{Slice}
	\end{lemma}
	\begin{proof}[Proof of Proposition \ref{princi}]
        First of all, we will prove that $ \mu \llcorner_{\partial^* E} $ is a rectifiable Radon measure. It is not difficult to see that $ \mu \ll \mathcal{H}^n $. Indeed, let $ A \subset \R^n \times \R $ be a set with $ \mathcal{H}^n ( A ) = 0 $, and let the set $ D_k := \{ ( x , t ) \in \R^n \times \R^+ \mid \Theta^{*n} ( \mu , ( x , t ) ) \leq k \} $ for each $ k \in \N $. By \cite[Theorem 3.2]{simon1983lectures}, we have
        \[
            \mu ( A \cap D_k ) \leq 2^n\,k\,\mathcal{H}^n( A \cap D_k ) = 0
        \]
        for all $ k \in \N $. Furthermore, by the upper bound of mass density ratio (Lemma \ref{upperdensityratio}), we see that $ \mu ( A \backslash \bigcup_{k=1}^{\infty} D_k ) = 0 $. Thus we obtain $ \mu ( A ) = 0 $, that is, $ \mu \ll \mathcal{H}^n $ holds. Since $ \mu \ll \mathcal{H}^n $, $ | \nabla^{\prime } \chi_E | = \mathcal{H}^n \llcorner_{\partial^* E} $ and Proposition \ref{S<<dmudtprop}, we see that
        \[
            \mu \llcorner_{\partial^* E} \ll | \nabla^{\prime} \chi_E | , \quad | \nabla^{\prime} \chi_E | \ll \mu \llcorner_{\partial^* E}.
        \]
        By Radon--Nikod\'{y}m theorem, there exists a function $f=(d\mu\llcorner_{\partial^* E} )/d|\nabla'\chi_E|$ with $0<f<\infty$ 
        for $|\nabla'\chi_E|$-a.e., $ f \in L_{loc}^1 (| \nabla^{\prime} \chi_E |) $ and 
        $\mu\llcorner_{\partial^* E}=f|\nabla'
        \chi_E|=f\mathcal H^n\llcorner_{\partial^* E}$. This shows that $\mu\llcorner_{\partial^* E}$ is a rectifiable Radon measure and the tangent
        space $T_{(x,t)}(\mu\llcorner_{\partial^* E})$ with multiplicity $f$ exists for $\mathcal H^n$-a.e.\,$(x,t)\in \partial^* E\cap\{t>0\}$. For the next step, we prove that $ T_{(x,t)}\,\mu = T_{(x,t)} (\partial^* E) $ for $ \mathcal{H}^n $-a.e.\,$ ( x, t ) \in \partial^* E \cap \{ t > 0 \} $. Now, by \cite[Theorem 3.5]{simon1983lectures}, we see that
        \[
            \limsup_{r\to+0} \frac{ \mu (B^{n+1}_r(x,t)) \setminus \partial^* E ) }{ r^n } = 0 \quad \text{for $ \mathcal{H}^n $-a.e.\,} (x,t) \in \partial^* E \cap \{ t > 0 \}.
        \]
        Let then $ \phi \in C^0_c (B^{n+1}_1(0)) $ be arbitrary, we have
        \begin{equation*}
        \begin{split}
            \lim_{r\to+0} &\left| \int_{\R^n\times(0,\infty)\setminus\partial^* E} \frac{1}{r^n}\,\phi\left( \frac{1}{r} (y-x,s-t) \right) d\mu(y,s)\,\right|\\
            &\leq \| \phi \|_{C^0} \limsup_{r\to+0} \frac{ \mu (B^{n+1}_r(x,t)) \setminus \partial^* E ) }{ r^n } = 0
    \end{split}
        \end{equation*}
        for $ \mathcal{H}^n $-a.e.\,$ ( x, t ) \in \partial^* E \cap \{ t > 0 \} $. Thus, by $ f \in L_{loc}^1 (| \nabla^{\prime} \chi_E |) $, we obtain at each Lebesgue point of $f$
        \begin{equation*}
            \begin{split}
                \lim_{r\to+0} \int_{\R^n\times(0,\infty)} \frac{1}{r^n}\,\phi&\left( \frac{1}{r} (y-x,s-t) \right) d\mu(y,s)\\
                &= \lim_{r\to+0} \int_{\partial^* E} \frac{1}{r^n}\,\phi\left( \frac{1}{r} (y-x,s-t) \right)\,\frac{ d \mu }{ d | \nabla^{\prime} \chi_E | } (y,s)\ d\mathcal{H}^n(y,s)\\
                &= 
                f(x,t) \int_{T_{(x,t)} (\partial^* E)} \phi(y,s)\ d\mathcal{H}^n (y,s)
            \end{split}
        \end{equation*}
        for all $ \phi \in C^0_c (\R^n\times\R) $ and $ \mathcal{H}^n $-a.e.\,$ ( x, t ) \in \partial^* E \cap \{ t > 0 \} $. This completes the proof of $ T_{(x,t)}\,\mu = T_{(x,t)}(\partial^* E) $.
        \vskip.5\baselineskip
        By Proposition \ref{L2cor}, Proposition \ref{BtoL2}, and the above argument, (1) and (2) are proved. Next, we prove (3) and (4). By Lemma \ref{Slice}, we see that the following for a.e.\,$ t > 0 $ and $ \mathcal{H}^{n-1} $-a.e.\,$ x \in (\partial^* E)_t $:
        \begin{align}
            &\mathcal{H}^{n-1} (\partial^* E_t \Delta (\partial^* E)_t) = 0, \label{Sliceprop1}\\
            &\mathbf{p}(\nu_E (x,t)) \neq 0 ,\label{Sliceprop2}\\
            &\nu_{E_t}(x) = \frac{\mathbf{p}(\nu_E(x,t))}{|\mathbf{p}(\nu_E(x,t))|}. \label{Sliceprop3}
        \end{align}
        Let $ A := \{ t > 0 \mid \text{ (\ref{Sliceprop1}) fails} \} $ and for every $ t > 0 $ set $ A_t := \{ x \in (\partial^* E)_t \mid x \notin \partial^* E_t \text{ or (\ref{Sliceprop2})-(\ref{Sliceprop3}) fail} \} $, so that $ \mathcal{L}^1 (A) = 0 $ and $ \mathcal{H}^{n-1} (A_t) = 0 $ for every $ t \in (0,\infty) \setminus A $. Consider then the characteristic function $ \chi ( x , t ) := \chi_{A_t} (x) $ on $ \R^n \times (0,\infty) $, since $ \mathcal{L}^1 (A) = 0 $ and $ \mathcal{H}^{n-1} (A_t) = 0 $ for every $ t \in (0,\infty) \setminus A $, we have
        \begin{equation*}
            \begin{split}
                \int_{\partial^* E} \chi (x,t)\,|\nabla^{\partial^* E} (\mathbf{q}(x,t))|\ d\mathcal{H}^{n} (x,t)
                &= \int_0^{\infty} \int_{(\partial^* E)_t} \chi (x,t)\ d\mathcal{H}^{n-1} dt\\
                &= \int_0^{\infty} \mathcal{H}^{n-1}(A_t)\ dt = \int_A \mathcal{H}^{n-1} (A_t)\ dt = 0,
            \end{split}
        \end{equation*}
        where we used the co-area formula in the first line, and where $ \nabla^{ \partial^* E } $ is the gradient on the tangent plane of $ \partial^* E $, that is,
        \[
        \nabla^{ \partial^* E } \mathbf{ q } ( x , t ) = P_{T_{(x,t)}(\partial^* E)} (\nabla \mathbf{ q } ( x , t ) ).
        \]
        Here, combining (1) and (2), we see that
        \[
            \begin{pmatrix}
				h ( x , t )\\
				1
			\end{pmatrix} \in T_{(x,t)} (\partial^* E) \quad \text{at $ \mathcal{H}^n $-a.e.\,$ (x,t) \in \partial^* E \cap \{ t>0 \} $},
        \]
        which implies $ |\nabla^{\partial^* E} (\mathbf{q}(x,t))| > 0 $ for $ \mathcal{H}^n $-a.e.\,$ (x,t) \in \partial^* E \cap \{ t>0 \} $. Hence, it must be $ \chi (x,t) = 0 $ for $ \mathcal{H}^n $-a.e.\,$ (x,t) \in \partial^* E \cap \{ t>0 \} $, thus the first part of (3) and (4) is proved. For the proof of the identity $ T_x\,\mu_t = T_x (\partial^* E_t) $, it is obtained by Lemma \ref{E<<mu}, Lemma \ref{upperdensityratio} and repeating the argument of (1) at fixed $ t $.
        \vskip.5\baselineskip
        Finally, we prove (5). Taking the $ (x,t) \in \partial^* E $ as satisfying (1)-(4) of this Proposition, we can calculate as
        \[
            {}^t ( z , 0 ) \cdot \nu_E ( x , t )
            = z \cdot \mathbf{p}( \nu_E(x,t) )
            = | \mathbf{p} (\nu_E(x,t)) |\,(z\cdot\nu_{E_t}(x))
            = 0
        \]
        for all $ z \in T_x(\partial^* E_t) $. This completes the proof of (5).
        \end{proof}

	\subsection{Boundaries move by mean curvature}
        In this subsection, we prove Theorem \ref{mainresult} by rephrasing the velocity $ V $ in Proposition \ref{velocity} as the mean curvature 
        and by using geometric measure theory. The argument for this rephrasing corresponds to the proof of the area formula (\ref{theareachange}). Since Definition \ref{generalizedBVflow}(ii) is treated in Remark \ref{abstcontirmk}, we can deduce that $ ( \{ \mu_t \}_{ t \in \R^+ } , \{ E_t \}_{ t \in \R^+ } ) $ as in Section \ref{translative} is a generalized BV flow.
        \vskip.5\baselineskip
	\begin{proof}[Proof of Theorem \ref{mainresult}]
		We fix a test function $ \phi \in C^1_c ( \R^n \times ( 0 , \infty ) ) $ arbitrarily. Then, by using Gauss-Green's theorem for set of finite perimeter, we have
		\begin{equation}
			\int_{ \R^n \times ( 0 , \infty ) } \partial_t \phi\, \chi_E\ dx dt = \int_{ \partial^* E } \phi\,\mathbf{ q } ( \nu_{ E } )\, d\mathcal{ H }^n. \label{Dt1S}
		\end{equation}
		Let $ G $ be the set satisfying Proposition \ref{princi}(1)-(5). Then for all $ ( x , t ) \in G $, we have
		\begin{equation}
            \label{span}
			T_{ ( x , t ) }\, \mu = ( T_x ( \partial^* E_t ) \times \{ 0 \} ) \oplus \mathrm{ span }
			\begin{pmatrix}
				h ( x , t )\\
				1
			\end{pmatrix} \quad (\text{by Proposition \ref{princi}(2)}).
		\end{equation}
		By $ h ( x , t ) \perp T_x \, \mu_t $, (\ref{span}), Proposition \ref{princi}(1) and (4), we have
		\begin{equation}
			\nu_{ E } ( x , t ) = \frac{ 1 }{ \sqrt{ 1 + | h ( x , t ) |^2 } }
			\begin{pmatrix}
				\nu_{ E_t } ( x )\\
				- h ( x , t ) \cdot \nu_{ E_t } ( x )
			\end{pmatrix}.
		\end{equation}
		According to (\ref{span}) and $ T_{ ( x , t ) }\, \mu = T_{ ( x , t ) } ( \partial^* E ) $ on $ G $, we obtain that the co-area factor of the projection ${\bf q}$ satisfies
		\begin{equation}
			| \nabla^{ \partial^* E } \mathbf{ q } ( \nu_E ( x , t ) ) | = \frac{ 1 }{ \sqrt{ 1 + | h ( x , t ) |^2 } },
			\label{areaelement}
		\end{equation}
        Due to (\ref{Dt1S})-(\ref{areaelement}) and the co-area formula, we compute as
		\begin{equation}
			\begin{split}
				&\int_{ \R^n \times ( 0 , \infty ) } \partial_t \phi\,\chi_E\ dxdt
				= - \int_{ G } \phi\,h \cdot \nu_{ E_t } \,\frac{ 1 }{ \sqrt{ 1 + | h |^2 } }\ d\mathcal{ H }^n
				= - \int_{ \partial^* E } \phi\,h \cdot \nu_{ E_t }\,| \nabla^{ \partial^* E } \mathbf{ q } ( \nu_E ) |\ d\mathcal{ H }^n\\
				&= - \int_{ 0 }^{ \infty } \int_{ \partial^* E \cap \{ \mathbf{ q } = t \} } \phi\,h \cdot \nu_{ E_t }\ d\mathcal{ H }^{ n - 1 } dt
				= - \int_{ 0 }^{ \infty } \int_{ \partial^* E_t } \phi\,h \cdot \nu_{ E_t }\ d\mathcal{ H }^{ n - 1 } dt, \label{GBV}
			\end{split}
		\end{equation}
		where we used $ \mathcal{ H }^n ( \partial^* E \setminus G ) = 0 $.
        By the same cut-off argument of Lemma \ref{areacontilemma} for (\ref{GBV}) and Remark \ref{abstcontirmk}, we deduce
		\[
		\int_{ E_{ t_2 } } \phi ( x , t_2 )\ dx - \int_{ E_{ t_1 } } \phi ( x , t_1 )\ dx = \int_{ t_1 }^{ t_2 } \int_{ E_t } \partial_t \phi\ dx dt + \int_{ t_1 }^{ t_2 } \int_{ \partial^* E_t } \phi\ h \cdot \nu_{ E_t }\ d\mathcal{ H }^{ n - 1 } dt
		\]
		for all $ 0 < t_1 < t_2 < \infty $. Use the continuity of Remark \ref{abstcontirmk}, we obtain the above equality for all $ 0 \leq t_1 < t_2 < \infty $ and all $ \phi \in C^1_c ( \R^n \times \R^+ ) $. This completes the proof.
	\end{proof}
    
	\appendix
	\section{Measure-Function Pairs}
	Here, we recall the notion of measure-function pairs introduced by Hutchinson in \cite{hutchinson1986second}.
	\begin{definition}
		Let $ E \subset \R^n $ be an open set and let $ \mu $ be a Radon measure on $ E $. Suppose $ f \in L^1 ( \mu ; \R^d ) $. Then we say that $ ( \mu , f ) $ is $ \R^d $-valued measure-function pair over $ E $.
	\end{definition}
	Next, we define the notion of convergence for a sequence of $ \R^d $-valued measure-function pairs over $ E $.
	\begin{definition}
		Let $ \{ ( \mu_i , f_i ) \}_{ i = 1 }^{ \infty } $ and $ ( \mu , f ) $ be $ \R^d $-valued measure-function pairs over $ E $. Suppose
		\[
		\mu_i \rightharpoonup \mu
		\] 
		as Radon measure on $ E $. Then we call $ ( \mu_i , f_i ) $ converges to $ ( \mu , f ) $ in the weak sense if
		\[
		\int_E f_i \cdot \phi\ d \mu_i \to \int_E f \cdot \phi\ d \mu
		\]
		for all $ \phi \in C^0_c ( E ; \R^d ) $.
	\end{definition}
	We present a less general version of \cite[Theorem 4.4.2]{hutchinson1986second} to the extent that it can be used in this paper.
	\begin{theorem}
		\label{measurefunctiontheorem}
		Suppose that $ \R^d $-valued measure-function pairs $ \{ ( \mu_i , f_i ) \}_{ i = 1 }^{ \infty } $ are satisfied
		\[
		\sup_i \int_E | f_i |^2\ d \mu_i < \infty.
		\]
		Then the following hold:
		\begin{description}
			\item[\quad\textup{(1)}] There exist a subsequence $ \{ ( \mu_{ i_j } , f_{ i_j } ) \}_{ j = 1 }^{ \infty } $ and $ \R^d $-valued measure-function pair $ ( \mu , f ) $ such that $ ( \mu_{ i_j } , f_{ i_j } ) $ converges to $ ( \mu , f ) $ as measure-function pair.
			\item[\quad\textup{(2)}] If $ ( \mu_{ i_j } , f_{ i_j } ) $ converges to $ ( \mu , f ) $ then
			\[
			\int_E | f |^2\ d \mu \leq \liminf_{ j \to \infty } \int_E | f_{ i_j } |^2\ d \mu_{ i_j } < \infty.
			\]
		\end{description}
	\end{theorem}

	\bibliography{myref.bib}

\begin{thebibliography}{10}

\bibitem{bertini2017stochastic}
L.~Bertini, P.~Butt{\`a}, and A.~Pisante.
\newblock Stochastic {A}llen--{C}ahn approximation of the mean curvature flow:
  large deviations upper bound.
\newblock {\em Archive for Rational Mechanics and Analysis}, 224(2):659--707,
  (2017).

\bibitem{brakke1978motion}
K.~A. Brakke.
\newblock {\em The motion of a surface by its mean curvature}, volume~20 of
  {\em Mathematical notes}.
\newblock Princeton University Press, Princeton, 1978.

\bibitem{chen1991uniqueness}
Y.~G. Chen, Y.~Giga, and S.~Goto.
\newblock Uniqueness and existence of viscosity solutions of generalized mean
  curvature flow equations.
\newblock {\em Journal of Differential Geometry}, 33(3):749--786, (1991).

\bibitem{Edelen+2020+95+137}
N.~Edelen.
\newblock The free-boundary brakke flow.
\newblock {\em Journal f\"ur die reine und angewandte Mathematik (Crelles
  Journal)}, 2020(758):95--137, (2020).

\bibitem{evans1991motion}
L.~C. Evans and J.~Spruck.
\newblock Motion of level sets by mean curvature. {I}.
\newblock {\em Journal of Differential Geometry}, 33(3):635--681, (1991).

\bibitem{fischer2020local}
J.~Fischer, S.~Hensel, T.~Laux, and T.~Simon.
\newblock The local structure of the energy landscape in multiphase mean
  curvature flow: Weak-strong uniqueness and stability of evolutions.
\newblock {\em arXiv preprint arXiv:2003.05478}, (2020).

\bibitem{hutchinson1986second}
J.~E. Hutchinson.
\newblock Second fundamental form for varifolds and the existence of surfaces
  minimising curvature.
\newblock {\em Indiana University Mathematics Journal}, 35(1):45--71, (1986).

\bibitem{ilmanen1994elliptic}
T.~Ilmanen.
\newblock Elliptic regularization and partial regularity for motion by mean
  curvature.
\newblock {\em Memoirs of the American Mathematical Society}, 108(520), (1994).

\bibitem{kasai2014general}
K.~Kasai and Y.~Tonegawa.
\newblock A general regularity theory for weak mean curvature flow.
\newblock {\em Calculus of Variations and Partial Differential Equations},
  50(1):1--68, (2014).

\bibitem{laux2016convergence}
T.~Laux and F.~Otto.
\newblock Convergence of the thresholding scheme for multi-phase mean-curvature
  flow.
\newblock {\em Calculus of Variations and Partial Differential Equations},
  55(5):1--74, (2016).

\bibitem{laux2018convergence}
T.~Laux and T.~M. Simon.
\newblock Convergence of the {A}llen-{C}ahn equation to multiphase mean
  curvature flow.
\newblock {\em Communications on Pure and Applied Mathematics},
  71(8):1597--1647, (2018).

\bibitem{luckhaus1995implicit}
S.~Luckhaus and T.~Sturzenhecker.
\newblock Implicit time discretization for the mean curvature flow equation.
\newblock {\em Calculus of variations and partial differential equations},
  3(2):253--271, (1995).

\bibitem{maggi2012sets}
F.~Maggi.
\newblock {\em Sets of finite perimeter and geometric variational problems: an
  introduction to Geometric Measure Theory}, volume 135 of {\em Cambridge
  Studies in Advanced Mathematics}.
\newblock Cambridge University Press, Cambridge, 2012.

\bibitem{roger2008allen}
L.~Mugnai and M.~R{\"o}ger.
\newblock The {A}llen--{C}ahn action functional in higher dimensions.
\newblock {\em Interfaces and Free Boundaries}, 10(1):45--78, (2008).

\bibitem{SchulzeWhite+2020+281+305}
F.~Schulze and B.~White.
\newblock A local regularity theorem for mean curvature flow with triple edges.
\newblock {\em Journal f\"ur die reine und angewandte Mathematik (Crelles
  Journal)}, 2020(758):281--305, (2020).

\bibitem{simon1983lectures}
L.~Simon.
\newblock {\em Lectures on geometric measure theory}, volume~3 of {\em
  Proceedings of the Centre for Mathematical Analysis}.
\newblock Australian National University, Canberra, 1983.

\bibitem{stuvard2022endtime}
S.~Stuvard and Y.~Tonegawa.
\newblock End-time regularity theorem for {B}rakke flows.
\newblock {\em arXiv preprint, arXiv:2212.07727}, (2022).

\bibitem{StuvardTonegawa+2022}
S.~Stuvard and Y.~Tonegawa.
\newblock On the existence of canonical multi-phase {B}rakke flows.
\newblock {\em Advances in Calculus of Variations}, (2022).
\newblock Ahead of print, https://doi.org/10.1515/acv-2021-0093.

\bibitem{tonegawa2014second}
Y.~Tonegawa.
\newblock A second derivative {H}{\"o}lder estimate for weak mean curvature
  flow.
\newblock {\em Advances in Calculus of Variations}, 7(1):91--138, (2014).

\bibitem{tonegawa2019brakke}
Y.~Tonegawa.
\newblock {\em {B}rakke's {M}ean {C}urvature {F}low: {A}n {I}ntroduction}.
\newblock SpringerBriefs in Mathematics. Springer, Singapore, 2019.

\bibitem{brian2005regularity}
B.~White.
\newblock A local regularity theorem for mean curvature flow.
\newblock {\em Annals of Mathematics}, 161(3):1487--1519, (2005).

\bibitem{WhiteMCFlecture2015}
B.~White.
\newblock Mean curvature flow (math 258) lecture notes.
\newblock https://web.stanford.edu/\~{}ochodosh/MCFnotes.pdf, 2015.

\bibitem{WhiteMCF2021}
B.~White.
\newblock Mean curvature flow with boundary.
\newblock {\em Ars Inveniendi Analytica}, page~43, (2021).
\newblock arXiv:1901.03008.

\end{thebibliography}
	\bibliographystyle{abbrv}
	
	\hrulefill
	
\end{document}